\documentclass[11pt,reqno]{amsart}

\usepackage{amsmath, amsfonts, amsthm, amssymb,graphicx, color, cite,enumitem}
\usepackage[margin=1.2in]{geometry}
 
\usepackage{latexsym,hyperref}

\usepackage{hyperref}
\hypersetup{hidelinks}

\newtheorem{theorem}{Theorem}[section]
\newtheorem{proposition}{Proposition}[section]
\newtheorem{lemma}{Lemma}[section]

\theoremstyle{definition}

\theoremstyle{remark}
\newtheorem{remark}{Remark}[section]

\catcode`@=11 \@addtoreset{equation}{section} \catcode`@=12

\UseRawInputEncoding

\begin{document}
\title[Asymptotic behavior of solutions to a singular chemotaxis system]
{Asymptotic behavior of solutions to a singular chemotaxis system in multi-dimensions}

\author[Q. Tao]{Qiang Tao}
\address{School of Mathematical Sciences, Shenzhen University, Shenzhen 518060, China}
 \email{taoq060@126.com}
 
\author[D. Wang]{Dehua Wang}
\address{Department of Mathematics,   University of Pittsburgh,  Pittsburgh, PA 15260, USA}
 \email{dhwang@pitt.edu}
 
\author[Y. Yang]{Ying Yang}
\address{School of Mathematical Sciences, Shenzhen University, Shenzhen 518060, China}
 \email{yiyiying729@163.com}

\author[M. Zhong]{Meifang Zhong$^*$}
\address{School of Mathematical Sciences, Shenzhen University, Shenzhen 518060, China}
 \email{13560907864@163.com}
\thanks{$^*$Corresponding author}


\begin{abstract}
 In this paper, we investigate the optimal large-time behavior of the global solution to a singular chemotaxis system in the whole space $\mathbb{R}^d$ with $d=2,3$.
Assuming that the initial data is sufficiently close to an equilibrium state, we first prove the $k$-th order spatial derivative of the global solution converges to its corresponding equilibrium at the optimal rate  $(1+t)^{-(\frac{d}{4}+\frac{k}{2})}$, which improve upon the result in \cite{WXY2016}. 
Then, for well-chosen initial data, we also establish lower bounds on the convergence rates, which
match those of the heat equation. Our proof relies on a Cole-Hopf type transformation, delicate spectral analysis, the Fourier splitting technique, and energy methods.
 
 \end{abstract}



\date{\today}

\maketitle

\section{Introduction}

Chemotaxis describes the oriented movement of cells toward the chemical concentration gradient, which plays a significant role in various biological processes. To model this common biological phenomenon, Keller and Segel \cite{KS1970, KS1971} first developed the prototypical chemotaxis model.
 The extensively studied chemotaxis model is the following classical Keller-Segel system:
\begin{align*}
    \left\{
        \begin{array}{l}
        u_t=  \nabla\cdot (D\nabla u - \chi(c) u \nabla  c)  ,
        \\
        c_t =\varepsilon \Delta c - \mu \kappa(c) u ,
        \end{array}
    \right.
\end{align*}
 where $u(x,t)$ and $c(x,t)$ denote the density of cellular population and the concentration of chemical signal at position $x$  and time $t$, respectively.
The functions $\chi(c)$ and $\kappa(c)$ denote the chemotactic sensitivity and the consumption rate of the substrate by the cells, respectively.
The parameter $D>0$ is the diffusion rate of cells, and $\varepsilon\geq 0$ is the diffusion rate of chemical substances.
The solvability and stability of this system with various $\chi(c)$ and $\kappa(c)$ can be found in \cite{BBTW2015, HP2009, H2003, W2012} and the references therein.

In this paper, we consider the Keller-Segel system with $\chi(c)=\frac{1}{c}$ and $\kappa(c)=c$,   modeling the onset of tumor angiogenesis, in the following form, 
\begin{align}\label{1.1}
    \left\{
        \begin{array}{l}
        u_t= {\rm div} (D\nabla u - u \nabla \ln c),
        \\
        c_t =\varepsilon \Delta c - \mu uc,
        \end{array}
    \right.
\end{align}
in $\mathbb{R}^d \times (0, \infty)$ with $d=2,3$. This model was proposed in \cite{HA2000} to describe the main interaction between vascular endothelial growth factor (VEGF) and vascular endothelial cells.
In this particular case, $u(x,t)$ represents the density of vascular endothelial cells, and $c(x,t)$ is the concentration of VEGF.

 An intriguing aspect of model \eqref{1.1} is the presence of a singular logarithmic sensitivity function $\ln c$, which is singular at $c = 0$. This unique logarithmic sensitivity was initially introduced by Keller and Segel in their influential work \cite{EF1971}, where it was employed to describe the formation of traveling wave bands resulting from bacterial chemotaxis, as observed in Adler's experiments \cite{JA1966}. Mathematical derivation was subsequently presented in \cite{HA1997}, while \cite{YV2009} provided a biological rationale based on experimental measurements and model simulations. These results imply that logarithmic sensitivity has both mathematical and biological significance, despite the challenges of mathematical analysis and numerical computation. To overcome singularity, as done in \cite{LS1997, WH2008}, an effective approach is to apply the Hopf-Cole transformation as
$$
{\bf v}=- \nabla \ln c=-\frac{\nabla c}{c},
$$
which implies that the model \eqref{1.1} can be transformed into the following system,
\begin{equation}\label{1.2}
 	\left\{ \begin{aligned}
 		&u_t - \Delta  u  = \nabla\cdot  (u{\bf v})  , 	                       \\
 	    &{\bf v}_t -\varepsilon \Delta {\bf v} =\nabla(-\varepsilon |{\bf v}|^2 +u) ,   \\
 		&(u,{\bf v})(x,0)=(u_0,{\bf v}_0)(x) ,
     \end{aligned}
 	\right.  	
\end{equation}
where we take $D=\mu=1$ for simplicity.

To put our results into perspective, let us highlight some important progress on the system \eqref{1.2}.
When $\varepsilon = 0$, the system \eqref{1.2} is considered as a system of parabolic-hyperbolic balance laws and the mathematical analysis of this system is satisfactory.
Global well-posedness and long-time behavior with large initial data in one-dimensional space and with small initial data in the multi-dimensional case were proved in \cite{DL2014, FFH2002, FZ2012, LLZ2011, LPZ2012}, respectively. For the results of decay rates,
in \cite{DLRHPKZ2015}, the authors discuss the Cauchy problem of the system \eqref{1.2} in one-dimensional space.
Under the small energy of the low frequency part of the initial data,   the algebraic decay rate was obtained as,
$$
\sum_{k=1}^3\big(t^k\|\partial_x^{k-1}(u(t)-\bar{u})\|^2_{L^2}+\|\partial_x^{k-1}v(t)\|^2_{L^2}
+\int^t_{\tilde{T}}\tau^k\|\partial_x^ku(\tau)\|^2_{L^2}d\tau\big)
+\sum_{l=1}^2\int_{\tilde{T}}^t\tau^l\|\partial_x^lv(\tau)\|^2_{L^2}d\tau\le \zeta_0,
$$
where $\tilde{T}>0$ is a finite time and $\zeta_0>0$ is a constant. For $d$-dimensional cases ($d\geq4$), Guan, Wang and Xu \cite{XYGSLWYLFYX2016} established the global existence and uniqueness of the strong solutions. Then, based on the a priori estimates, the authors obtained the convergence rates of the strong solutions in the critical Besov spaces as
$$
\|(v, u-\bar{u})\|_{B_{2,1}^{\frac{d}{2}-2}}\le C_0(1+t)^{-\frac{d}{4}}.
$$
Using the embedding theorem $B_{2,1}^{\frac{d}{2}-2}\hookrightarrow L^2$ for $d\geq4$, one concludes that
the convergence rates of the strong solutions in $L^2$ are optimal.
Then, in \cite{FYXXLLCLW2019}, the authors extended the decay rates of the solutions in $L^2$-critical regularity framework for $d\geq2$ as,
\begin{align*}
(1+t)^{-\frac{d}{2}-\frac{s}{2}}&\le(1+t)^{-\frac{d}{2}+\frac{1}{2}},\quad\text{for low frequencies},
\\
(1+t)^{-\frac{d}{2}-\frac{1}{2}+\varepsilon}&\le(1+t)^{-\frac{d}{2}+\frac{1}{2}},\quad\text{for high frequencies},
\end{align*}
for $s\geq-1$. Recently, some studies have focused on decay rates for the chemotaxis model with logistic growth in the one-dimensional space, as in \cite{YNZKZ2020, YNZKZ2022}. Further work has been done on traveling waves, shock waves, and boundary layers in many papers, such as \cite{CKKV2020, HLWW2018, JLW2013, PW2018, PWZZ2018, WH2008}.

Mathematical theory on the solvability and stability of \eqref{1.2} with $\varepsilon > 0$ has attracted a lot of attention
in the last decade. From the structure of the transformed system \eqref{1.2}, the parameter $\varepsilon$ is the coefficient of diffusion and convection, which plays a dual role.
Different types of equations exhibit quite different properties, such as Lyapunov functions and the stability of traveling waves. Therefore, the case $\varepsilon > 0$ could not be considered a simple extension of the case $\varepsilon = 0$.
The research on the one-dimensional version with $\varepsilon > 0$ is relatively well-developed. 
The global well-posedness and the exponential stability of classical solutions on a bounded domain with large initial data and the Dirichlet or Neumann-Dirichlet boundary conditions were proved in
\cite{LPZ2012, LZ2015, TWW2013, WZ2013}. Then, in \cite{MWZ2018}, the authors studied the qualitative behavior of classical solutions to the Cauchy problem. They obtained the global asymptotic stability of constant ground states and identified the explicit decay rate of solutions.
There are more works on the formation of shock waves, the nonlinear stability of traveling waves, and boundary layer analysis for $\varepsilon>0$, see, e.g., \cite{HWZ2016, LLW2014, LW2011, LW2012}.
Compared to the one-dimensional case, due to the energy criticality and super-criticality, the mathematical theory for the multi-dimensional
case with $\varepsilon > 0$ is far from satisfactory. Winkler \cite{W2016} addressed global weak solutions and large-time behavior in the two-dimensional space for arbitrarily large initial data. In a radially symmetric setting, a global renormalized solution has been constructed in \cite{W2018}
without any restriction on the spatial dimensions or the size of the initial data.
Meanwhile, there are a few works on the Cauchy problem for \eqref{1.2} with $\varepsilon > 0$ and the initial data close to the constant ground state $(\bar{u}, 0)$.
In \cite{PWZ2014}, the global well-posedness of classical solutions for this system was proved, provided that the smooth initial data has a suitably small $L^2$ norm. Subsequently, the authors in \cite{WXY2016} proved that the strong solutions exist globally
in time by the method of energy estimates, when the initial data $(u_0-\bar{u}, {\bf v}_0)\in H^k$ with $k\geq2 $ was taken to be properly small in $H^1$. In addition, if the initial data further belongs to a homogeneous negative index Sobolev space $\dot{H}^{-s}$ for $s\in (0,\frac{d}{2})$,
they established the following asymptotic decay rates of solutions,
\begin{align}
	\begin{aligned} \label{W}
		&\| \nabla ^l(u-\overline{u},{\bf v})(t) \| _{L^2} \leq C(1+t)^{-\frac{s+l}{2}}, \quad  l = 0,1,\cdots  k-1,\\
&\| \nabla ^k(u-\overline{u},{\bf v})(t) \| _{L^2} \leq C(1+t)^{-\frac{s+k-1}{2}}.
\end{aligned}
\end{align}
However, it is clear that the decay rate of the highest $k$-th order derivative of the solution coincides with the lower one $k-1$-th order, which does not seem to be the optimal decay rate.
Then, the global well-posedness and long-time behavior of classical solutions to the Cauchy problem and initial boundary value problem with small initial entropy are shown in \cite{RWWZZ2019} and \cite{WWZ2021}, respectively.
It is worth noting that the works \cite{PWZ2014,  RWWZZ2019, WWZ2021, WXY2016} also derived the convergence rates of \eqref{1.2} toward the non-diffusive system ($\varepsilon=0$).

In this paper, we focus on the asymptotic dynamics of the singular chemotaxis system in multi-dimensions with $\varepsilon \geq 0$.
On the one hand, we investigate the optimal decay rates of the global solution and its derivatives within classical Sobolev spaces. In particular, we determine the optimal decay rate of the critical spatial derivative of the global solution, which improves the $k$-th order spatial derivative result of the global solution presented in  \cite{WXY2016}. On the other hand, we establish the lower bounds of these decay rates to demonstrate their optimality.
 
First of all, let us state the result of the global well-posedness for the system \eqref{1.2} that has already been established in \cite{WXY2016}.
\begin{proposition}[\!\!\cite{WXY2016}]
    \label{prop} Let $(u_0-\bar{u},{\bf v}_0) \in H^k(\mathbb{R}^d)$ with some integer $k\geq 2$ for some constant background state $\bar u >0.$
	Then for any constant $M_0>0$ with $\|\nabla^2 u_0\|_{L^2(\mathbb{R}^d)}^2 + \|\nabla^2 {\bf v}_0\|_{L^2(\mathbb{R}^d)}^2 \leq M_0^2, $ there exists a positive constant $\kappa$ depending on $M_0$
	such that if
    \begin{align}\label{1.0}
    \|u_0 - \bar{u}\|_{H^1(\mathbb{R}^d)}^2 + \|{\bf v}_0\|_{H^1(\mathbb{R}^d)}^2 \leq \kappa^2,
    \end{align}
   the system \eqref{1.2} admits a unique global solution $(u,{\bf v})\in C([0,+\infty), H^k({\mathbb{R}^d}))$ satisfying
    \begin{align}
     \begin{aligned}
        \|u(t)&-\bar{u}\|_{H^k(\mathbb{R}^d)}^2 + \|{\bf v}(t)\|_{H^k(\mathbb{R}^d)}^2\\
        &+\int_{0}^{t}{\big (}\|\nabla u(\tau)\|_{H^k(\mathbb{R}^d)}^2 + \|\nabla {\bf v}(\tau)\|_{H^{k-1}(\mathbb{R}^d)}^2 +  \|\nabla^{k+1} {\bf v}(\tau)\|_{L^2(\mathbb{R}^d)}^2{\big )}d\tau\\
        \leq & C(\|u_0 - \bar{u}\|_{H^k(\mathbb{R}^d)}^2 + \|{\bf v}_0\|_{H^k(\mathbb{R}^d)}^2 )
    \end{aligned}
     \end{align}
   for all $t \geq 0,$ where $C$ is a positive constant independent of $\kappa$, $t$ and $\varepsilon$.
\end{proposition}

The main novelty of our paper is to study the optimal decay rates of the global solution, including their upper bounds and lower bounds.
 The main results are stated in the following theorems.


\begin{theorem}[Upper bounds]
    \label{th} 	
    Assume that, in addition to all of the conditions in Proposition \ref{prop},  the initial data $(u_0-\bar{u}, {\bf v}_0)\in L^1(\mathbb{R}^d)  \cap H^N(\mathbb{R}^d)$ for some integer $N \geq 2$. Then
    the solution $(u-\bar{u},{\bf v})$ to the system \eqref{1.2} with suitably small $\kappa$ in \eqref{1.0} has the following decay rates,
    \begin{align}\label{thn}
    	\Vert \nabla ^k(u-\bar{u},{\bf v})(t) \Vert _{L^2(\mathbb{R}^d)} \leq C(1+t)^{-\frac{d+2k}{4}},
    \end{align}
    for all $t \geq 0$ and $ k=0, 1, \ldots, N$, where $C$ is a positive constant independent of time $t$ and $\varepsilon$.
\end{theorem}

\begin{remark}
\label{remark 1.4}
It follows from Theorem \ref{th} that the following convergence rates for the critical spatial derivative ($N$-th order) hold true:
$$	
\| \nabla ^N(u-\bar{u},{\bf v})(t) \| _{L^2} \leq C(1+t)^{-\frac{s+N}{2}},
$$
which improves and complements the $N$-order decay rates in \cite{WXY2016} as \eqref{W} with $s=\frac{d}{2}$. Moreover, the convergence rates are the same as those of the solutions to heat equations.
\end{remark}

\begin{theorem}[Lower bounds]
\label{th1}

 Assume that,  in addition to all of the conditionss in Theorem \ref{th}, 
\begin{align}\label{-1}
\int_{{\mathbb{R}^d}}(u_0-\bar{u}, {\bf v}_0)dx \neq0.
\end{align}
Then there exists a suitable small $\kappa$ in \eqref{1.0}, such that the solution $(u-\bar{u},{\bf v})$ to the system \eqref{1.2} exhibits decay rates as follows
	\begin{align}\label{thn2}
				\Vert \nabla ^k(u-\bar{u},{\bf v})(t) \Vert _{L^2(\mathbb{R}^d)}\geq {\tilde{c}}(1+t)^{-\frac{d+2k}{4}},  \quad for \quad k=0, 1, \ldots, N,
		\end{align}
	where $\tilde{c}$ is a positive constant  independent of time $t$ and $\varepsilon$.
\end{theorem}

\begin{remark}
\label{remark 1.5}
The condition \eqref{-1} is quite important in the proof of the lower bound for the linearized problem. As pointed out in \cite{CPT2021}, this kind of condition is weaker than
those in most of the previous results, where the differentiability of the Fourier transform of
the initial disturbance is required in general.
\end{remark}

Finally, as a corollary, the results can be transferred back to the original system \eqref{1.1}.

\begin{theorem} \label{th2} Assume $(u_0, \nabla{\rm ln }c_0) \in L^1(\mathbb{R}^d)  \cap H^N(\mathbb{R}^d)$ for some integer $N \geq 2$ with $(u_0,c_0)(x)=(u,c)(x,0)$ satisfying the compatibility condition ${\bf v}_0=-\nabla{\rm ln }c_0.$ Then for any constant $M_0 > 0$ with $\|\nabla^2 u_0\|_{L^2(\mathbb{R}^d)}^2 + \|\nabla^3 {\rm ln }c_0\|_{L^2(\mathbb{R}^d)}^2 \leq M_0^2,$ there exists a positive constant $\kappa$ depending on $M_0$ such that if
		$$\|u_0-\bar{u}\|_{H^1(\mathbb{R}^d)}^2 + \|\nabla{\rm ln }c_0\|_{H^1(\mathbb{R}^d)}^2 \leq \kappa^2,$$
		for some constant $\bar{u}\geq 0$, the system \eqref{1.1} admits a unique global classical solution.
Furthermore, there is a positive constant $C$ such that
		$$\|u-\bar{u}\|_{L^\infty(\mathbb{R}^d)} \leq C(1+t)^{-\frac{d+2}{4}},$$
		$$\|c\|_{L^\infty(\mathbb{R}^d)} \leq Ce^{-\bar{u}t},$$
where $C$ is independent of time $t$.
	
\end{theorem}

Now, let us outline the strategy for proving our results. The proof can be divided into four steps. 
First, we introduce a perturbation form of the system \eqref{1.2}. Then, by utilizing the representation of the Green's function and spectral analysis of the semigroup to the
corresponding linear system, we obtain the $L^2$ time decay rates, including both upper and lower bounds, for the linearized system.
Second, we derive the upper bounds for the optimal decay rates of the nonlinear system \eqref{1.2} within the classical Sobolev space $H^k$ framework, which differs from the previous works for $\varepsilon=0$ in \cite{DL2014} and for $\varepsilon\geq 0$ in \cite{WXY2016}.
These spaces help us to obtain the convergence rates for the critical spatial derivative (the $N$-th order) and enable us to deduce lower bounds from the upper bounds of the optimal decay rates, which match those of the heat equation.
 To determine the decay rate of higher-order spatial derivatives over time, we apply Fourier splitting methods. Specifically, by focusing on the time integrability of the density's dissipative term rather than directly absorbing it into the dissipation itself, we identify the critical derivative influencing the decay rate. 
 Third, we establish the lower bounds on the convergence rates of the nonlinear system. To do so, we use the decay rate of the linear semigroup, Duhamel's principle, and the upper bounds of the optimal decay rates. 
 Finally, we transfer these results back to the original system  \eqref{1.1} using the Cole-Hopf transformation.


The paper is organized as follows. Section 2 introduces basic inequalities and reformulates the system \eqref{1.1}. 
Section 3 presents a detailed spectral analysis of the semigroup and establishes the $L^2$ decay rates for the linearized system. 
Section 4 is dedicated to the optimal time decay rates of the global solution for the nonlinear system. 
In the final section, the convergence results of the $L^{\infty}$ norm for the global solution to the original system \eqref{1.1} are proved based on the transformed system.


{\bf Notation:} Throughout this paper, we use $H^k(\mathbb{R} ^d)(k \in \mathbb{R} , d=2,3)$ to denote the classical Sobolev spaces with norm $\| \cdot \Vert _{H^k}$
and $L^p(\mathbb{R} ^d)(1 \leq p \leq \infty)$ to denote the usual $L^p$ spaces with norm $\| \cdot   \Vert  _{L^p}$. 
The symbol $\nabla ^l$ with an integer $l \geq 0$ stands for the usual any spatial derivatives of order $l$. 
For instance, we define $\nabla ^l z = \left\{\partial _x ^\alpha z_i: | \alpha | = l, i=1,2 \right\}, z=(z_1,z_2)£¬$ in 2D and $\nabla ^l z = \left\{\partial _x ^\alpha z_i: | \alpha | = l, i=1,2,3 \right\}, z=(z_1,z_2,z_3)$ in 3D. For the sake of simplicity, we write $\int {f} dx := \int _{\mathbb{R} ^d}{f} dx$
and $\| (A,B)\|_X := \|A\|_X + \|B\|_X .$

\bigskip

\section {Preliminaries}

In this section, let us first recall some elementary inequalities that will be used extensively in later sections.

\begin{lemma}[Sobolev interpolation inequality]
\label{lemma 2.1}
  Let $2 \leq p \leq +\infty$ and $0 \leq l,k \leq m$. If $p = +\infty$,  we require furthermore that
   $l \geq k+1.$   
   Then, if $\nabla ^l u \in L^2(\mathbb{R}^d)$ and $ u \in L^2(\mathbb{R}^d)$, we have
   $\nabla ^k u \in L^p(\mathbb{R}^d)$. Moreover, there exists a positive constant $C$ dependent on $k, l, m, p$ such that
   \begin{align*}
       \Vert \nabla ^k u \Vert _{L^p} \leq C\Vert \nabla ^l u \Vert _{L^2}^\theta  \Vert u \Vert _{L^2}^{1-\theta },
   \end{align*}
   where $0 \leq \theta \leq 1$ satisfying
     $$\frac{1}{p} =\frac{k}{d}+\Big(\frac{1}{2}-\frac{l}{d}\Big)\theta + {\frac{1}{2}}{(1-\theta )}.$$
\end{lemma}

For the following commutator estimate, we refer to \cite{AA2002}.
\begin{lemma} \label{lemma 2.2}
	Let $k \geq 1$ be an integer and define the commutator
	$${[\nabla ^k, f]g} := 
	\nabla ^k{(fg)} - f\nabla ^k g.$$
	Then we have
	$$\Vert {[\nabla ^k, f]g} \Vert_{L^2} \leq C\Vert \nabla f \Vert _{L^\infty}\Vert \nabla ^{k-1}g \Vert_{L^2} + C\Vert \nabla ^k f \Vert _{L^2} \Vert g \Vert_{L^\infty} ,$$
	where $k$ is the sole variable that affects the positive constant $C$.
\end{lemma}

Then, we provide the reformulation of the system \eqref{1.2}. Without loss of generality, we set $\bar{u}=1$. By setting $n=u-1$ and $n_0=u_0-1,$ we rewrite \eqref{1.2} in the following perturbation form:
\begin{equation}
\label{yuan}
	\left\{
	\begin{array}{l}
		n_t - \Delta n - {\rm div}{\bf v} = S_1, \\
		{\bf v}_t -  \varepsilon\Delta{\bf v} - \nabla n = S_2,\\
		(n,{\bf v})(x,0)=(n_0,{\bf v}_0)(x),
	\end{array}
	\right.
\end{equation}
where $S_1$ and $S_2$ are defined by
\begin{equation}\label{yuan2}
	\begin{aligned}
		S_1={\rm div} (n{\bf v}) ,\qquad
		S_2= - \varepsilon\nabla | {\bf v}|^2.
	\end{aligned}
\end{equation}


By some straightforward calculations, one can derive the following a priori estimates, which will play an important role in establishing the time decay rates of solutions:
\begin{equation}
\label{S}
\begin{aligned}
	\|(S_1,S_2)\|_{L^1(\mathbb{R}^d)} &\leq \|(n, {\bf v})\|_{L^2(\mathbb{R}^d)} \|(\nabla n, \nabla {\bf v})\|_{L^2(\mathbb{R}^d)}
	                                  \leq\kappa \|(\nabla n, \nabla {\bf v})\|_{L^2(\mathbb{R}^d)},\\
	\|(S_1,S_2)\|_{L^2(\mathbb{R}^d)} &\leq (\| n\|_{H^1(\mathbb{R}^d)}^\frac{1}{2}+\|{\bf v}\|_{H^1(\mathbb{R}^d)}^\frac{1}{2})(\|\nabla^2 n\|_{L^2(\mathbb{R}^d)}^\frac{1}{2}+\|\nabla^2 {\bf v}\|_{L^2(\mathbb{R}^d)}^\frac{1}{2})\|(\nabla n, \nabla{\bf v})\|_{L^2(\mathbb{R}^d)}\\
	                                  &\leq\kappa^\frac{1}{2} M_0^\frac{1}{2} \|(\nabla n, \nabla{\bf v})\|_{L^2(\mathbb{R}^d)},\\
	\|\nabla(S_1,S_2)\|_{L^2(\mathbb{R}^d)} &\leq (\|\nabla n\|_{L^2(\mathbb{R}^d)}^\frac{1}{2}\|\nabla^2 n\|_{L^2(\mathbb{R}^d)}^\frac{1}{2} + \|\nabla {\bf v}\|_{L^2(\mathbb{R}^d)}^\frac{1}{2}\|\nabla^2 {\bf v}\|_{L^2(\mathbb{R}^d)}^\frac{1}{2})\|(\nabla^2n, \nabla^2{\bf v})\|_{L^2(\mathbb{R}^d)}\\
                         	                &\leq\kappa^\frac{1}{2} M_0^\frac{1}{2}\|(\nabla^2n, \nabla^2{\bf v})\|_{L^2(\mathbb{R}^d)}.
\end{aligned}
\end{equation}

\bigskip

\section{Spectral Analysis and Linear $L^2$ Estimates}

In this section, we apply spectral analysis to the semigroup of the linearized system \eqref{yuan} and then prove time-decay estimates for it.

\subsection{Spectral analysis of the semigroup}

Let us consider the following initial value problem for the linearized system of \eqref{yuan}.
\begin{equation}\label{3.2}
	\left\{
	\begin{aligned}
		&\widetilde{n}_t - \Delta \widetilde{n} - {\rm div} \widetilde{{\bf v}} =0,   \qquad  && x\in \mathbb{R} ^d,\ t > 0,\\
		&\widetilde{{\bf v}}_t -  \varepsilon\Delta \widetilde{{\bf v}} - \nabla \widetilde{n} = 0,  \qquad  && x\in \mathbb{R} ^d,\ t > 0,\\
		&(\widetilde{n},\widetilde{{\bf v}})(x,0)=(\widetilde{n}_0,\widetilde{{\bf v}}_0)(x),                 \qquad  && x\in \mathbb{R} ^d.
	\end{aligned}
	\right.
\end{equation}
Utilizing semigroup theory, we can represent the solution $(\widetilde{n},\widetilde{{\bf v}})$ of the linearized equation \eqref{3.2} in form of
 $\widetilde{U}=(\widetilde{n},\widetilde{{\bf v}})^t$. In fact, the system \eqref{3.2} can be rewritten as
$$
\widetilde{U}_t = B\widetilde{U}, \qquad t \geq 0, \qquad \widetilde{U}(0) = \widetilde{U}_0,
$$
which gives rise to
$$
\widetilde{U}(t) = S(t)\widetilde{U}(0) = e^ {tB}\widetilde{U}(0), \qquad t \geq 0,
$$
where
$$
B = \left(  {\begin{array}{cc}
		\Delta & \rm div\\
		\nabla &  \varepsilon\Delta
\end{array}}\right).
$$
Now, our goal is to analyze the differential operator $B$ through its Fourier expression $A(\xi)$ and show the long-time behavior of the semigroup $S(t)$. Applying a Fourier transform to system \eqref{3.2}, we obtain
\begin{align*}
\partial_t \widehat{\widetilde{U}}(t,\xi) &= A(\xi)\widehat{\widetilde{U}}(t,\xi), \  t\geq 0,
\\
\widehat{\widetilde{U}}(\xi,0)&=\widehat{\widetilde{U}}_0(\xi),
\end{align*}
where $\xi = (\xi_1,\xi_2)^t$ in 2D and $\xi = (\xi_1,\xi_2,\xi_3)^t$ in 3D,  and $A(\xi)$ is defined as
$$ A(\xi) = \left(  {\begin{array}{cc}
		-|\xi|^2 & i\xi^t\\
		i\xi & -\varepsilon|\xi|^2I_{d\times d}
\end{array}}\right).
$$
Since
$${\rm det}(A(\xi)-\lambda I) = (\lambda +  \varepsilon|\xi|^2)^2(\lambda ^2 + (\varepsilon +1)|\xi|^2\lambda +  \varepsilon|\xi|^4 + |\xi|^2)=0,
$$
we can get the eigenvalues
$$
\lambda_0 = -\varepsilon|\xi|^2~ {\rm (double)},  \qquad \lambda_+=\lambda_+(|\xi|), \qquad \lambda_-=\lambda_-(|\xi|).
$$
Therefore, the semigroup $e^{At}$ can be shown as
 $$
 e^{At} = e^{\lambda_0 t}P_0 +  e^{\lambda_+ t}P_ + +  e^{\lambda_- t}P_-,
 $$
where the projector $P_0$, $P_+$ and $P_-$ can be expressed as
 $$
 P_i = \prod _{i\neq j} \frac{A(\xi)-\lambda_j I}{\lambda_i -\lambda_j}, \qquad i,j=0,+,-.
 $$
By some straightforward calculations, one can establish the precise formulation for the Fourier-transformed Green's function $\widehat{G}(t,\xi)$ corresponding to $G(t,x)=e^{tB}$ as
\begin{align}
\begin{aligned}
& \widehat{G}(t,\xi) = e^{tA} \\=& \left(  {\begin{array}{cc}
		\frac{\lambda_+ e^{\lambda_- t} - \lambda_- e^{\lambda_+ t}}{\lambda_+ - \lambda_-} - \frac{ e^{\lambda_+t} - e^{\lambda-t}}{\lambda_+ - \lambda_-}|\xi|^2  & \frac{i \xi^t (e^{\lambda_+t} - e^{\lambda_-t})}{\lambda_+ - \lambda_-}\\
		\frac{i \xi (e^{\lambda_+t} - e^{\lambda_-t})}{\lambda_+ - \lambda_-} & e^{\lambda_0 t}(I- \frac{\xi \xi^t}{|\xi|^2})
		+ \frac{\xi \xi^t}{|\xi|^2}(\frac{\lambda_+ e^{\lambda_- t} - \lambda_- e^{\lambda_+ t}}{\lambda_+t - \lambda_-t}
		- \varepsilon|\xi|^2\frac{e^{\lambda_+t} - e^{\lambda_-t}}{\lambda_+ - \lambda_-})
\end{array}}\right).
\end{aligned}
\end{align}
Denote $$\widehat{G}(t,\xi) =\left(  {\begin{array}{cc} \widehat{N} \\ \widehat{M} \end{array}}\right)$$ and divide the convolution $(\widetilde{n},\widetilde{{\bf v}} )= G \ast \widetilde{U}_0$  into the frequency components
$$
\widehat{\widetilde{n}} = \widehat{N} \cdot \widehat{\widetilde{U}}_0,
\quad\text{and}\quad
\widehat{\widetilde{{\bf v}} }= \widehat{M} \cdot \widehat{\widetilde{U}}_0.
$$
In order to obtain the decay rate of the solution, we need to estimate lower frequency parts and high frequency parts for
the Fourier transform $\widehat{G}(t,\xi)$ of the Green  function.
Based on the eigenvalue definitions, for $\xi \leq \eta$, we can deduce
\begin{equation}\label{3.3}
	\begin{aligned}
		\lambda_+ = -\frac{(1 +\varepsilon)|\xi|^2}{2} + \frac{i}{2}\sqrt{4(|\xi|^2+|\xi|^4) - (1+ \varepsilon)^2|\xi|^4} = a+bi,\\
		\lambda_- = -\frac{(1 +\varepsilon)|\xi|^2}{2} - \frac{i}{2}\sqrt{4(|\xi|^2+|\xi|^4) - (1+\varepsilon)^2|\xi|^4} = a-bi,
	\end{aligned}
\end{equation}
and then,
\begin{equation}
\label{3.4}
	\begin{aligned}
		&\frac{\lambda_+ e^{\lambda_- t} - \lambda_- e^{\lambda_+ t}}{\lambda_+ - \lambda_-} = e^{-\frac{1}{2}(1 +\varepsilon)|\xi|^2t}[\cos(bt) + \frac{1}{2}(1 +\varepsilon)|\xi|^2 \frac{\sin(bt)}{b}],\\			
		&\frac{ e^{\lambda_+t} - e^{\lambda_-t}}{\lambda_+ - \lambda_-} = e^{-\frac{1}{2}(1 +\varepsilon)|\xi|^2t}\frac{\sin(bt)}{b}.				\end{aligned}
\end{equation}
Here we approximate $b$ via Taylor expansion around low frequency parts as
$$
b=\frac{1}{2}\sqrt{4(|\xi|^2 + |\xi|^4) - (1+\varepsilon)^2|\xi|^4} \sim |\xi| + O(|\xi|^3), \quad |\xi| \leq \eta.
$$
Similarly, for high frequency parts for $\xi \geq \eta$, we  get
\begin{equation}
	\begin{aligned}
		\lambda_+ = -\frac{(1 +\varepsilon)|\xi|^2}{2} -\frac{1}{2}\sqrt{(1+\varepsilon)^2|\xi|^4 - 4(|\xi|^2 + |\xi|^4)  } = a-b,\\
		\lambda_- = -\frac{(1 +\varepsilon)|\xi|^2}{2} + \frac{1}{2}\sqrt{(1+\varepsilon)^2|\xi|^4 - 4(|\xi|^2 + |\xi|^4) } = a+b,
	\end{aligned}
\end{equation}
which leads
\begin{equation*}
	\begin{aligned}
		&\frac{\lambda_+ e^{\lambda_- t} - \lambda_- e^{\lambda_+ t}}{\lambda_+ - \lambda_-} = \frac{1}{2}e^{(a+b)t}[1+e^{-2bt}]- \frac{a}{2b}e^{(a+b)t}[1-e^{-2bt}] \sim O(1)e^{-R_0t}, \quad \xi \geq \eta,\\			
		&\frac{ e^{\lambda_+t} - e^{\lambda_-t}}{\lambda_+ - \lambda_-} = \frac{1}{2b}e^{(a+b)t}[1-e^{-2bt}] \sim O(1)\frac{1}{|\xi|^2}e^{-R_0t}, \quad \xi \geq \eta,			
	\end{aligned}
\end{equation*}
where $b$ can be approximated as
 $$
 b=\frac{1}{2}\sqrt{(1+\varepsilon)^2|\xi|^4 - 4(|\xi|^2 + |\xi|^4)  } \sim \frac{\varepsilon+1}{2}|\xi|^2 - \frac{2}{\varepsilon+1} + O(|\xi|^{-2}), \quad \xi \geq \eta,
 $$
and $R_0$ and $\eta$ are some positive constants.

\subsection{Linear $L^2$ estimates for system \eqref{3.2}}

By virtue of the Green's function in frequency space and the asymptotic analysis above, we can obtain the following result concerning the long-time properties of the global solutions of the linearized system in $L^2$ norm.
\begin{lemma}
\label{U}
Assume $U_0 = (n_0,{\bf v}_0) \in L^1(\mathbb R^d) \cap H^l(\mathbb R^d) (l \geq 2)$.
The global classical solution  $(\tilde{n},\tilde{{\bf v}}) \in C([0,\infty);H^N(\mathbb R^d))$ to the linearized system \eqref{3.2} satisfies
$$
\|\nabla ^k (\widetilde{n},\widetilde{{\bf v}})(t)\|_{L^2(\mathbb R^d)} \leq C(1+t)^{-\frac{d}{4}-\frac{k}{2}}(\|U_0\|_{L^1(\mathbb R^d)} + \|\nabla^kU_0 \|_{L^2(\mathbb R^d)}),
$$
with $0\leq k \leq N$. Here the positive constant $C$ is independent of $t$.
\end{lemma}

\begin{proof} 
By a direct computation, combined with the expression of the Green's function $\widehat{G}(t,\xi)$, one obtains
\begin{align*}
	\widehat{\widetilde{n}}(t,\xi)
	=& \frac{\lambda_+ e^{\lambda_- t} - \lambda_- e^{\lambda_+ t}}{\lambda_+ - \lambda_-}\widehat{n}_0 - \frac{ e^{\lambda_+t} - e^{\lambda-t}}{\lambda_+ - \lambda_-}|\xi|^2\widehat{n}_0 + \frac{i \xi^t (e^{\lambda_+t} - e^{\lambda_-t})}{\lambda_+ - \lambda_-}\widehat{{\bf v}}_0\\
	\begin{split}
		&\sim\left\{
		\begin{array}{ll}
			O(1)e^{-\frac{1}{2}(1+\varepsilon)|\xi|^2t}(|\widehat{n}_0|+|\hat{{\bf v}}_0|), \quad |\xi|\leq \eta,\\
			O(1)e^{-R_0t}(|\widehat{n}_0|+|\widehat{{\bf v} }_0|), \quad |\xi|\geq \eta,\\
		\end{array}
		\right.
	\end{split}	
\end{align*}
\begin{align*}
	&\widehat{\widetilde{\bf v}}(t,\xi)\\
	=&\frac{i \xi (e^{\lambda_+t} - e^{\lambda_-t})}{\lambda_+ - \lambda_-} \widehat{n}_0 + e^{\lambda_0 t}(I- \frac{\xi \xi^t}{|\xi|^2})\widehat{\bf v}_0
	+ \frac{\xi \xi^t}{|\xi|^2}\bigg(\frac{\lambda_+ e^{\lambda_- t} - \lambda_- e^{\lambda_+ t}}{\lambda_+t - \lambda_-t}\widehat {\bf v}_0
	-\varepsilon |\xi|^2\frac{e^{\lambda_+t} - e^{\lambda_-t}}{\lambda_+ - \lambda_-}\widehat{\bf v}_0\bigg)\\
	\begin{split}
		&\sim\left\{
		\begin{array}{ll}
			O(1)e^{-\varepsilon|\xi|^2t}(|\widehat{n}_0|+|\widehat{\bf v}_0|), \quad |\xi|\leq \eta,\\
			O(1)e^{-R_0t}(|\widehat{n}_0|+|\widehat{\bf v}_0|), \quad |\xi|\geq \eta.\\
		\end{array}
		\right.
	\end{split}	
\end{align*}
Here $R_0$ and $\eta$ are some constants. Then we establish the  $L^2$-decay rate for $(\widetilde{n},\widetilde{{\bf v}})$ as
\begin{equation*}
	\begin{aligned}
		&\|(\widehat{\widetilde{n}},\widehat{\widetilde{{\bf v}}})(t)\|_{L^2(\mathbb R^d)}^2\\
		=&\int_{|\xi|\leq\eta}|(\widehat{\widetilde{n}},\widehat{\widetilde{{\bf v}}})(t,\xi)|^2 d\xi + \int_{|\xi|\geq\eta}|(\widehat{\widetilde{n}},\widehat{\widetilde{{\bf v}}})(t,\xi)|^2 d\xi\\
		\leq &C\int_{|\xi|\leq\eta}e^{-(1+\varepsilon)|\xi|^2t}(|\widehat{n}_0|^2+|\widehat{{\bf v}}_0|^2)d\xi + \int_{|\xi|\geq\eta} e^{-R_0t}(|\widehat{n}_0|^2+|\widehat{{\bf v}}_0|^2)d\xi\\
		\leq &C(1+t)^{-\frac{d}{2}}\|(n_0,{\bf v}_0)\|_{L^1(\mathbb{R}^d) \bigcap L^2(\mathbb{R}^d)}^2.
	\end{aligned}
\end{equation*}
Furthermore, the $L^2$-decay rate for the derivatives of $(\nabla^k\widetilde{n},\nabla^k\widetilde{\bf v})$ is deduced by
\begin{equation*}
	\begin{aligned}
		&\|(\widehat{\nabla^k\widetilde{n}},\widehat{\nabla^k\widetilde{{\bf v}}})(t)\|_{L^2(\mathbb R^d)}^2\\
		=&\int_{|\xi|\leq\eta}|\xi|^{2k}|(\widehat{\widetilde{n}},\widehat{\widetilde{{\bf v}}})(t,\xi)|^2 d\xi + \int_{|\xi|\geq\eta}|\xi|^{2k}|(\widehat{\widetilde{n}},\widehat{\widetilde{{\bf v}}})(t,\xi)|^2 d\xi\\
		\leq &C\int_{|\xi|\leq\eta}e^{-(1+\varepsilon)|\xi|^2t}|\xi|^{2k}(|\widehat{n}_0|^2+|\widehat{{\bf v}}_0|^2)d\xi + \int_{|\xi|\geq\eta} e^{-R_0t}|\xi|^{2k}(|\widehat{n}_0|^2+|\widehat{{\bf v}}_0|^2)d\xi\\
		\leq &C(1+t)^{-\frac{d}{2}-k}(\|(n_0,{\bf v}_0)\|_{L^1(\mathbb{R}^d)}^2 + \|(\nabla^kn_0,\nabla^k {\bf v}_0)\|_{ L^2(\mathbb{R}^d)}^2 ).
	\end{aligned}
\end{equation*}
The proof of this lemma is completed.
\end{proof} 

\begin{lemma}
\label{linear}
Under the assumptions of Theorem \ref{th1},  there is a unique global classical solution  $(\tilde{n},\tilde{{\bf v}}) \in C([0,\infty);H^N(\mathbb R^d))$ to the linearized system \eqref{3.2}, satisfies the inequality:
\begin{equation}
\label{nv}
\tilde{c}(1+t)^{-\frac{d}{4}-\frac{k}{2}} \leq\|\nabla ^k (\widetilde{n},\widetilde{{\bf v}})\|_{L^2(\mathbb R^d)} \leq C(1+t)^{-\frac{d}{4}-\frac{k}{2}},
\end{equation}
where $0\leq k \leq N$, $\tilde{c}$ and $C$ are positive constants  independent of time $t$.
\end{lemma}

\begin{proof}
For simplicity, we first present the case $k=0$. The decay rates for higher-order derivatives can be obtained similarly. By the formula of Green's function $\widehat{G}(t,\xi)$, we deduce that
\begin{equation*}
	\begin{aligned}
		\widehat{\widetilde{n}}(t,\xi)
		=& \frac{\lambda_+ e^{\lambda_- t} - \lambda_- e^{\lambda_+ t}}{\lambda_+ - \lambda_-}\widehat{n}_0 - \frac{ e^{\lambda_+t} - e^{\lambda-t}}{\lambda_+ - \lambda_-}|\xi|^2\widehat{n}_0
		+\frac{i \xi^t (e^{\lambda_+t} - e^{\lambda_-t})}{\lambda_+ - \lambda_-}\widehat{{\bf v}}_0\\
		=&e^{-\frac{1}{2}(\varepsilon +1)|\xi|^2t}[\cos(bt)\widehat{n}_0  + i \xi^t \frac{\sin(bt)}{b}\widehat{{\bf v}}_0]
		-e^{-\frac{1}{2}(\varepsilon+1)|\xi|^2t}\frac{\sin(bt)}{b} |\xi|^2\widehat{n}_0\\
		&+ \frac{1}{2}(\varepsilon+1)e^{-\frac{1}{2}(\varepsilon +1)|\xi|^2t}\frac{\sin(bt)}{b} |\xi|^2\widehat{n}_0\\
		=&e^{-\frac{1}{2}(\varepsilon+1)|\xi|^2t}[\cos(bt)\widehat{n}_0  + i \xi^t \frac{\sin(bt)}{b}\widehat{{\bf v}}_0] 	+ e^{-\frac{1}{2}(1 +\varepsilon)|\xi|^2t}[\frac{1}{2}(\varepsilon-1) |\xi|^2\frac{\sin(bt)}{b}\widehat{n}_0 ]\\
		:=&T_1 + T_2,  \qquad {\text{for}} \quad |\xi|\leq \eta,
	\end{aligned}
\end{equation*}
\begin{equation*}
	\begin{aligned}
		&\widehat{\widetilde{{\bf v}}}(t,\xi)\\
		=&\frac{i \xi (e^{\lambda_+t} - e^{\lambda_-t})}{\lambda_+ - \lambda_-} \widehat{n}_0 + e^{\lambda_0 t}(I- \frac{\xi \xi^t}{|\xi|^2})\widehat{{\bf v}}_0
		+ \frac{\xi \xi^t}{|\xi|^2}(\frac{\lambda_+ e^{\lambda_- t} - \lambda_- e^{\lambda_+ t}}{\lambda_+t - \lambda_-t}\widehat{{\bf v}}_0
		- |\xi|^2\frac{e^{\lambda_+t} - e^{\lambda_-t}}{\lambda_+ - \lambda_-}\widehat{{\bf v}}_0)\\
		=&\frac{i \xi (e^{\lambda_+t} - e^{\lambda_-t})}{\lambda_+ - \lambda_-} \widehat{n}_0 + e^{\lambda_0 t}(I- \frac{\xi \xi^t}{|\xi|^2})\widehat{{\bf v}}_0
		+ \frac{\lambda_+ e^{\lambda_- t} - \lambda_- e^{\lambda_+ t}}{\lambda_+t - \lambda_-t} \frac{\xi (\xi \cdot \widehat{{\bf v}}_0)}{|\xi|^2}\\
		&- |\xi|^2 \frac{e^{\lambda_+t} - e^{\lambda_-t}}{\lambda_+ - \lambda_-}\frac{\xi (\xi \cdot \widehat{{\bf v}}_0)}{|\xi|^2}\\
		=&e^{-\frac{1}{2}(1 +\varepsilon)|\xi|^2t}\frac{\sin(bt)}{b}i \xi\widehat{n}_0 + e^{\lambda_0 t}(I- \frac{\xi \xi^t}{|\xi|^2})\widehat{{\bf v}}_0
		+ e^{-\frac{1}{2}(1 +\varepsilon)|\xi|^2t}[\cos(bt) \\
		&\quad+  \frac{1}{2}(1 +\varepsilon)|\xi|^2 \frac{\sin(bt)}{b}]\frac{\xi (\xi \cdot \widehat{{\bf v}}_0)}{|\xi|^2}
		 -   \varepsilon|\xi|^2 e^{-\frac{1}{2}(\varepsilon +1)|\xi|^2t} \frac{\sin(bt)}{b}\frac{\xi (\xi \cdot \widehat{{\bf v}}_0)}{|\xi|^2}\\
		=& {e^{-\frac{1}{2}(1 +\varepsilon)|\xi|^2t}[\cos(bt)\frac{\xi (\xi \cdot \widehat{{\bf v}}_0)}{|\xi|^2} + \frac{\sin(bt)}{b}i \xi\widehat{n}_0]
			+ e^{\lambda_0 t}(I- \frac{\xi \xi^t}{|\xi|^2})\widehat{{\bf v}}_0 }
		\\
&+ e^{-\frac{1}{2}(1 +\varepsilon)|\xi|^2t}[\frac{1}{2}(1-\varepsilon) |\xi|^2\frac{\sin(bt)}{b} \frac{\xi (\xi \cdot \widehat{{\bf v}}_0)}{|\xi|^2}] \\
		:=& D_1 +  D_2, \qquad {\text{for}} \quad |\xi|\leq \eta,
	\end{aligned}
\end{equation*}
where  $\eta$ is a sufficiently small constant. Then, the following can be derived,
\begin{equation}
\label{3.5}
	\begin{aligned}
		\|\widehat{\widetilde{n}}(t,\xi)\|_{L^2(\mathbb R^d)}^2 &= \int_{|\xi|\leq \eta}|\widehat{\widetilde{n}}(t,\xi)|^2d\xi + \int_{|\xi|\geq \eta}|\widehat{\widetilde{n}}(t,\xi)|^2d\xi\\
		&\geq \int_{|\xi|\leq \eta}|T_1 + T_2|^2d\xi \geq  \int_{|\xi|\leq \eta}\frac{1}{2}|T_1|^2 -|T_2|^2d\xi .
	\end{aligned}
\end{equation}
Therefore, one deduces that
\begin{equation}\label{3.6}
	\begin{aligned}
		\int_{|\xi|\leq \eta}|T_2|^2 d\xi &\leq C\|\hat{n}_0\|_{L^\infty(\mathbb{R}^d)}^2 \int_{|\xi|\leq \eta} e^{-(1 +\varepsilon)|\xi|^2t}|\xi|^4\big(\frac{\sin(bt)}{b}\big)^2d\xi\\
		&\leq C\|\hat{n}_0\|_{L^\infty(\mathbb{R}^d)}^2 \int_{|\xi|\leq \eta} e^{-(1 +\varepsilon)|\xi|^2t}|\xi|^2d\xi\\
		&\leq C(1+t)^{-(\frac{d}{2}+1)}\|n_0\|_{L^1(\mathbb{R}^d)}^2.
	\end{aligned}
\end{equation}
Since  $n_0(x) \in L^1$,  it implies $\widehat{n}_0(\xi)\in C(\mathbb{R}^d)$.  And if $\widehat{n}_0(0):=M_n = \int  n_0(x) dx \neq 0$,
it follows that, for a sufficiently small $\eta$, $\widehat{n}_0(\xi) \neq 0$ for $|\xi| \leq \eta$.
Therefore, when $M_n \neq 0,$ it is observed that
\begin{equation*}
	|\widehat{n}_0(\xi)|^2 \geq \frac{1}{C}|\int  n_0(x) dx |^2 \geq \frac{M_n^2}{C}, \quad {\text{for}} \quad |\xi| \leq \eta.
\end{equation*}
For the initial velocity field $\widehat{{\bf v}}_0$, a similar argument indicates that when $M_v=\int  {\bf v}_0(x) dx $ is non-zero, the following inequality holds
\\
$$
\frac{|\xi \cdot \widehat{{\bf v}}_0(\xi)|^2}{|\xi|^2} \geq \frac{|\xi \cdot M_v|^2}{C|\xi|^2}, \quad {\text{for}}\ |\xi| \leq \eta.
$$
\\
{\bf Case 1: } When $M_n \neq 0, M_v \neq 0$, employing the previous estimates and recognizing that $b \sim |\xi| + O(|\xi|^3)$ for $|\xi| \leq \eta,$ we have
\begin{equation}
\label{3.7}
	\begin{aligned}
		\int_{|\xi|\leq \eta}|T_1|^2 d\xi &\geq  \frac{M_n^2}{C} \int_{|\xi|\leq \eta} e^{-(1 +\varepsilon)|\xi|^2t}\cos^2(bt)d\xi
		+ \frac{1}{C}\int_{|\xi|\leq \eta} \frac{|\xi \cdot M_v|^2}{b^2} e^{-(\varepsilon +1)|\xi|^2t}\sin^2(bt)d\xi\\
		&\geq \frac{\min\{M_n^2,M_v^2\}}{C}\int_{|\xi|\leq \eta} e^{-(1 +\varepsilon)|\xi|^2t}(\cos^2(bt) + \sin^2(bt))d\xi\\
		&\geq C_1\int_{|\xi|\leq \eta} e^{-(1 +\varepsilon)|\xi|^2t}d\xi\\
		&\geq \tilde{c}(1+t)^{-\frac{d}{2}}.
	\end{aligned}
\end{equation}
{\bf Case 2: } When $M_n \neq 0, M_v = 0,$ and noticing the continuity of $\widehat{{\bf v}}_0$ near $\xi =0,$ it is assumed that there exists a sufficiently small constant $\epsilon$ such that $\epsilon \rightarrow 0$ as
$\xi \rightarrow 0.$ Under these conditions, we have
\begin{equation*}
	|\widehat{m}_0(\xi)|^2 < \epsilon, \quad {\text{for}} \quad |\xi| \leq \eta.
\end{equation*}
Letting $r = |\xi|\sqrt{t}$, one can show that
\begin{align}\label{3.8}
	\begin{aligned}
		\int_{|\xi|\leq \eta}|T_1|^2 d\xi &\geq  \frac{M_n^2}{C} \int_{|\xi|\leq \eta} e^{-(1 +\varepsilon)|\xi|^2t}\cos^2(bt)d\xi
		-\frac{\epsilon}{C}\int_{|\xi|\leq \eta}  e^{-(1 +\varepsilon)|\xi|^2t}\sin^2(bt)d\xi\\
		&\geq  \frac{M_n^2}{C}t^{-\frac{d}{2}}\int_{0}^{\eta\sqrt{t}}e^{-(1 +\varepsilon)r^2}\cos^2(r\sqrt{t})r^2dr\\
		&\quad -\frac{\epsilon}{C}t^{-\frac{d}{2}}\int_{0}^{\eta\sqrt{t}}e^{-(1 +\varepsilon)r^2}\sin^2(r\sqrt{t})r^2dr\\
		&\geq \frac{M_n^2}{C}t^{-\frac{d}{2}} \sum_{k=0}^{[\frac{\eta t}{\pi}]-1}\int_{\frac{k\pi}{\sqrt{t}}}^{\frac{k \pi+\frac{\pi}{4}}{\sqrt{t}}}
		e^{-(1 +\varepsilon)r^2}\cos^2(r\sqrt{t})r^2dr - \frac{\epsilon}{C}(1+t)^{-\frac{d}{2}}\\
		&\geq \frac{M_n^2}{C}t^{-\frac{d}{2}} \sum_{k=0}^{[\frac{\eta t}{\pi}]-1}\int_{\frac{k\pi}{\sqrt{t}}}^{\frac{k \pi+\frac{\pi}{4}}{\sqrt{t}}}
		e^{-(1 +\varepsilon)r^2}\frac{1+\cos(2r\sqrt{t})}{2}r^2dr - \frac{\epsilon}{C}(1+t)^{-\frac{d}{2}}\\
		&\geq \frac{M_n^2}{2C}t^{-\frac{d}{2}} \sum_{k=0}^{[\frac{\eta t}{\pi}]-1}\int_{\frac{k\pi}{\sqrt{t}}}^{\frac{k \pi+\frac{\pi}{4}}{\sqrt{t}}}
		e^{-(1 +\varepsilon)r^2}r^2dr - \frac{\epsilon}{C}(1+t)^{-\frac{d}{2}}\\
		&\geq C_1^{-1}(1+t)^{-\frac{d}{2}}-C_2^{-1}\epsilon(1+t)^{-\frac{d}{2}}\\
		&\geq \tilde{c}(1+t)^{-\frac{d}{2}}.
	\end{aligned}
\end{align}
{\bf Case 3: }  When $M_n = 0, M_v \neq 0,$ we have
\begin{equation}\label{3.9}
	\begin{aligned}
		\int_{|\xi|\leq \eta}|T_1|^2 d\xi &\geq  -\frac{\epsilon}{C} \int_{|\xi|\leq \eta} e^{-(1 +\varepsilon)|\xi|^2t}\cos^2(bt)d\xi
		+\frac{M_v^2}{C}\int_{|\xi|\leq \eta}  e^{-(1 +\varepsilon)|\xi|^2t}\sin^2(bt)d\xi\\
		&\geq \tilde{c}(1+t)^{-\frac{d}{2}}.
	\end{aligned}
\end{equation}
Combining  \eqref{3.5}, \eqref{3.7}, \eqref{3.8} with \eqref{3.9}, we obtain a lower bound of the time decay rate for $\widetilde{n}(t,x)$ as
\begin{equation}
\label{n}
	\|{\widetilde{n}}(t,x)\|_{L^2(\mathbb R^d)}^2 = \|\widehat{\widetilde{n}}(t,\xi)\|_{L^2(\mathbb R^d)}^2 \geq \tilde{c}(1+t)^{-\frac{d}{2}}.
\end{equation}
Similarly, we can get the lower bound of the time decay rate for $\widetilde{{\bf v}}(t,x)$ as
\begin{equation}
\label{3.10}
\begin{aligned}
\|\widehat{\widetilde{{\bf v}}}(t,\xi)\|_{L^2(\mathbb R^d)}^2 \geq \int_{|\xi|\leq \eta}( \frac{1}{2}|D_1|^2 -|D_2|^2)d\xi,
\end{aligned}
\end{equation}
and we have
\begin{equation}
\label{3.11}
\begin{aligned}
\int_{|\xi|\leq \eta}|D_2|^2 d\xi \leq C(1+t)^{-(\frac{d}{2}+1)}\|{\bf v}_0\|_{L^1(\mathbb{R}^d)}^2.
\end{aligned}
\end{equation}
Taking the definition of $D_1$ into account, we get
\begin{equation*}
\begin{aligned}
		&\int_{|\xi|\leq \eta}|D_1|^2 d\xi \\
		\geq &\Big\{\frac{M_n^2}{C}\int_{|\xi|\leq \eta}  \frac{|\xi|^2}{b^2}e^{-(1 +\varepsilon)|\xi|^2t}\sin^2(bt)d\xi \\
		&\quad + \frac{1}{C}\int_{|\xi|\leq \eta} \frac{|\xi \cdot M_v^2|^2}{|\xi|^2} e^{-(1 +\varepsilon)|\xi|^2t}\cos^2(bt)d\xi\Big\}\\
		&\quad + \int_{|\xi|\leq \eta} e^{-(1 +\varepsilon)|\xi|^2t}\cos(bt) \frac{\xi (\xi \cdot \widehat{{\bf v}}_0)}{|\xi|^2}(\widehat{{\bf v}}_0 - \frac{\xi (\xi \cdot \widehat{{\bf v}}_0)}{|\xi|^2})d\xi \\
		:= &J_1 + J_2.
\end{aligned}
\end{equation*}
Direct computation leads to
\begin{equation}
\label{3.12}
J_1 \geq C^{-1}(1+t)^{-\frac{d}{2}}, \qquad J_2 = 0.
\end{equation}
Combining the estimates \eqref{3.10}, \eqref{3.11} and \eqref{3.12},  we obtain a lower bound for the time decay rate of $\widetilde{{\bf v}}(t,x)$ as
\begin{equation}
\label{v}
	\|{\widetilde{{\bf v}}}(t,x)\|_{L^2(\mathbb{R}^d)}^2 = \|\widehat{\widetilde{{\bf v}}}(t,\xi)\|_{L^2(\mathbb{R}^d)}^2 \geq C^{-1}(1+t)^{-\frac{d}{2}}.
\end{equation}
From \eqref{n} and \eqref{v}, one has that
\begin{equation*}
	\min \{\| \widetilde{n}(t)\|_{L^2(\mathbb{R}^d)}, \| \widetilde{{\bf v}}(t)\|_{L^2(\mathbb{R}^d)}\} \geq \widetilde{c}(1+t)^{-\frac{d}{4}},
\end{equation*}
where $ \widetilde{c}$ is a positive constant depending on $M_n$ and $M_v$. Thanks to this and recalling Lemma \ref{U}, the conclusion of the lemma follows.
\end{proof} 

\bigskip

\section{Optimal Convergence Rates for the Nonlinear System}

In this section, we will establish the optimal convergence rates of the solutions to the initial value problem of the nonlinear system, as stated in Theorem \ref{th} and Theorem \ref{th1}.

\subsection{Upper bounds of the convergence rates}
The upper bounds of the convergence rates will be studied in this subsection. We will establish the upper bounds of the optimal decay rates for $k$-th order spatial derivative of the solution in Sobolev spaces with $k=0,1,\cdots,N-1$.

Let us begin with the following a priori estimate, which has been proven in \cite{WXY2016} based on
the energy method:  
simplicity.
\begin{lemma}[\!\! \cite{WXY2016}]
\label{lemma 3-3-1}
Assume the solution $(n,{\bf v})$ of the system \eqref{yuan} satisfies the  assumptions in
Proposition \ref{prop}. Then for all $t \in [0,T]$, it holds that
\begin{equation}
\label{L3.3}
	\begin{aligned}
		&\frac{d}{dt}(\|\nabla^{k} n\|_{L^2({\mathbb{R}}^d)}^2 +\|\nabla^{k+1} n\|_{L^2({\mathbb{R}}^d)}^2 +\|\nabla^{k} {\bf v}\|_{L^2({\mathbb{R}}^d)}^2+ \|\nabla^{k+1} {\bf v}\|_{L^2({\mathbb{R}}^d)}^2 )\\
		+&C(\|\nabla^{k+1} n\|_{L^2({\mathbb{R}}^d)}^2 + \|\nabla^{k+2} n\|_{L^2({\mathbb{R}}^d)}^2 + \|\nabla^{k}\nabla\cdot {\bf v}\|_{L^2({\mathbb{R}}^d)}^2 + \varepsilon \|\nabla^{k+2} {\bf v}\|_{L^2({\mathbb{R}}^d)}^2 )\leq 0,
	\end{aligned}
\end{equation}
for $k=0,1, \dots , N-1$, where $C$ is a positive constant independent of time $t$.
\end{lemma}

Now, we are ready to show the decay rates for the solution of the system \eqref{yuan}.
First of all, we will prove the following upper bounds of the decay rates for the solution. 

\begin{lemma}
\label{lemma 3-3-0}
Under all of the assumptions in Theorem \ref{th}, the global classical solution $(n,{\bf v})$ of the system \eqref{yuan} satisfies
\begin{align}
\label{0}
\|n\|_{H^{N-k}(\mathbb{R}^d)}^2 + \|{\bf v}\|_{H^{N-k}(\mathbb{R}^d)}^2 \leq C(1+t)^{-\frac{d}{2}-k},
\end{align}
for $t \geq 0$ and $k=0,1,$ where $C$ is a positive constant  independent of time $t$.
\end{lemma}

\begin{proof}  
Summing inequalities \eqref{L3.3} from $l=k$ to $l=m-1$, we get
\begin{align}
\label{leq1}
	\frac{d}{dt}\|\nabla^k (n,{\bf v})\|_{H^{m-k}(\mathbb{R}^d)}^2 + C\|\nabla^{k+1} n\|_{H^{m-k}(\mathbb{R}^d)}^2 + C\|\nabla^{k+1}{\bf v}\|_{H^{m-k-1}(\mathbb{R}^d)}^2 \leq 0.
\end{align}
Adding $\|\nabla^k(n,{\bf v})\|_{L^2(\mathbb{R}^d)}^2$ to two sides of \eqref{leq1}, it can be obtained that
\begin{align}
\label{+}
	\frac{d}{dt}\|\nabla^k (n,{\bf v})\|_{H^{m-k}(\mathbb{R}^d)}^2 + C\|\nabla^k (n,{\bf v})\|_{H^{m-k}(\mathbb{R}^d)}^2 \leq \|\nabla^k(n,{\bf v})\|_{L^2(\mathbb{R}^d)}^2.
\end{align}
 Let $\mathcal{F}_k^m  := \|\nabla^k (n,{\bf v})\|_{H^{m-k}(\mathbb{R}^d)}^2$, taking $m=N$ and $k=1$ in \eqref{+},  it holds that
 $$
 \frac{d}{dt} \mathcal{F}_1^N(t) + C\mathcal{F}_1^N(t) \leq \|\nabla(n,{\bf v})\|_{L^2(\mathbb{R}^d)}^2,
 $$
 and using the Gronwall's inequality, we deduce
\begin{align}
\label{F}
\mathcal{F}_1^N(t) \leq \mathcal{F}_1^N(0)e^{-Ct} + \int_{0}^{t}e^{-C(t-\tau)} \|\nabla(n,{\bf v})\|_{L^2(\mathbb{R}^d)}^2 d\tau.
\end{align}
In order to provide decay estimate of $\mathcal{F}_1^N(t)$, we need to discuss $\|\nabla(n,{\bf v})\|_{L^2(\mathbb{R}^d)}^2$. In fact, by the Duhamel's principle, the solution of \eqref{yuan}-\eqref{yuan2} can be expressed as
\begin{align}
\label{Du}
	(n,{\bf v})^{tr} = G(t)*(n_0,{\bf v}_0)^{tr} + \int_{0}^{t}G(t-\tau)*(S_1, S_2)^{tr}(\tau)d\tau.
\end{align}
Let $$F(t) = \displaystyle\sup_{\tau\in[0,t]}(1+\tau)^{\frac{d+2}{2}}(\|\nabla n(\tau)\|_{H^{N-1}(\mathbb{R}^d)}^2 + \|\nabla {\bf v}(\tau)\|_{H^{N-1}(\mathbb{R}^d)}^2).$$
 Using \eqref{S}, \eqref{Du} and Lemma \ref{U}, one can get
\begin{equation}
	\begin{aligned}
		\|\nabla(n,{\bf v})\|_{L^2(\mathbb{R}^d)}^2 &\leq C(1+t)^{-\frac{d+2}{2}} + C\int_0^t(\|(S_1,S_2)\|_{L^1(\mathbb{R}^d)}^2 \\
		&\quad + \|\nabla (S_1,S_2)\|_{L^2(\mathbb{R}^d)}^2)(1+t-\tau)^{-\frac{d+2}{2}}  d\tau\\
		&\leq C(1+t)^{-\frac{d+2}{2}} + C\int_0^t\delta(\|\nabla n\|_{H^1(\mathbb{R}^d)}^2 \\
		&\quad + \|\nabla{\bf v}\|_{H^1(\mathbb{R}^d)}^2)(1+t-\tau)^{-\frac{d+2}{2}}  d\tau\\
		&\leq C(1+t)^{-\frac{d+2}{2}} + C\delta F(t)\int_0^t(1+t-\tau)^{-\frac{d+2}{2}}(1+t)^{-\frac{d+2}{2}} d\tau\\
		&\leq C(1+t)^{-\frac{d+2}{2}} + C\delta F(t)(1+t)^{-\frac{d+2}{2}}.
	\end{aligned}
\end{equation}
Here, we utilized the following inequality:
\begin{equation*}
	\begin{aligned}
		&\quad\int_0^t(1+t-\tau)^{-\frac{d+2}{2}}(1+t)^{-\frac{d+2}{2}} d\tau\\
		&=\int_{0}^{\frac{t}{2}} (1+t-\tau)^{-\frac{d+2}{2}}(1+t)^{-\frac{d+2}{2}} d\tau
          + \int_{\frac{t}{2}}^{t}(1+t-\tau)^{-\frac{d+2}{2}}(1+t)^{-\frac{d+2}{2}} d\tau\\
		&\leq(1+\frac{t}{2})^{-\frac{d+2}{2}}\int_{0}^{\frac{t}{2}}(1+\tau)^{-\frac{d+2}{2}}d\tau
		+ (1+\frac{t}{2})^{-\frac{d+2}{2}}\int_{\frac{t}{2}}^{t}(1+t-\tau)^{-\frac{d+2}{2}}d\tau\\
		&\leq(1+t)^{-\frac{d+2}{2}}.
	\end{aligned}
\end{equation*}
Therefore, we obtain
\begin{align}
\label{1}
\|\nabla(n,{\bf v})\|_{L^2(\mathbb{R}^d)}^2 \leq C(1+t)^{-\frac{d+2}{2}}(1+\delta F(t)).
\end{align}
Substituting \eqref{1} into \eqref{F} gives that
\begin{equation}
\label{f}
	\begin{aligned}
		\mathcal{F}_1^N(t) &\leq C\mathcal{F}_1^N(0)e^{-Ct} + \int_{0}^{t}e^{-C(t-\tau)}(1+\tau)^{-\frac{d+2}{2}}(1+\delta F(\tau))d\tau\\
		&\leq C\mathcal{F}_1^N(0)e^{-Ct} + C(1+\delta F(t))\int_{0}^{t}e^{-C(t-\tau)}(1+\tau)^{-\frac{d+2}{2}}d\tau\\
		&\leq C\mathcal{F}_1^N(0)e^{-Ct} + C(1+\delta F(t))(1+t)^{-\frac{d+2}{2}}\\
		&\leq  C(1+\delta F(t))(1+t)^{-\frac{d+2}{2}},
	\end{aligned}
\end{equation}
with the help of the following inequality
\allowdisplaybreaks
\begin{equation*}
	\begin{aligned}
		&\quad\int_{0}^{t}e^{-C(t-\tau)}(1+\tau)^{-\frac{d+2}{2}}d\tau\\
		&=\int_{0}^{\frac{t}{2}}e^{-C(t-\tau)}(1+\tau)^{-\frac{d+2}{2}}d\tau
          + \int_{\frac{t}{2}}^{t}e^{-C(t-\tau)}(1+\tau)^{-\frac{d+2}{2}}d\tau\\
		&\leq e^{-\frac{C}{2}t}\int_{0}^{\frac{t}{2}}(1+\tau)^{-\frac{d+2}{2}}d\tau + (1+\frac{t}{2})^{-\frac{d+2}{2}}\int_{\frac{t}{2}}^{t}e^{-C(t-\tau)}d\tau\\
		&\leq C(1+t)^{-\frac{d+2}{2}}.
	\end{aligned}
\end{equation*}
Owing to the definition of $F(t)$ and \eqref{f}, it yields
$$
F(t)\leq C(1+\delta F(t)).
$$
At the same time, choosing  $\delta$ to be sufficiently small, we have
$$
F(t)\leq C.
$$
Therefore, we obtain the following decay estimate,
\begin{align}\label{1jie}
	\|\nabla n\|_{H^{N-1}(\mathbb{R}^d)}^2 + \|\nabla {\bf v}\|_{H^{N-1}(\mathbb{R}^d)}^2 \leq C(1+t)^{-\frac{d+2}{2}}.
\end{align}
On the other hand, using \eqref{S}, \eqref{Du}, \eqref{1jie} and Lemma \ref{U}, one has
\begin{equation}
	\begin{aligned}
			&\quad\|(n,{\bf v})\|_{L^2(\mathbb{R}^d)}^2\\
			&\leq C(1+t)^{-\frac{d}{2}} + C\int_{0}^{t}(\|(S_1,S_2)\|_{L^1(\mathbb{R}^d)}^2 + \|(S_1,S_2)\|_{L^2(\mathbb{R}^d)}^2)(1+t-\tau)^{-\frac{d}{2}}d\tau\\
			&\leq C(1+t)^{-\frac{d}{2}} + C\int_{0}^{t}\delta (\|\nabla n\|_{H^1(\mathbb{R}^d)}^2 + \|\nabla {\bf v}\|_{H^1(\mathbb{R}^d)}^2)(1+t-\tau)^{-\frac{d}{2}}d\tau\\
			&\leq C(1+t)^{-\frac{d}{2}} + C\int_{0}^{t}(1+t-\tau)^{-\frac{d+2}{2}}(1+\tau)^{-\frac{d}{2}}d\tau\\
			&\leq C(1+t)^{-\frac{d}{2}},
	\end{aligned}
\end{equation}
where we used the following inequality,
$$
\int_{0}^{t}(1+t-\tau)^{-\frac{d+2}{2}}(1+\tau)^{-\frac{d}{2}}d\tau\leq C(1+t)^{-\frac{d}{2}}.
$$
The proof of this lemma is completed.
\end{proof} 

Next, we will employ the Fourier splitting technique to improve the decay rates of higher-order spatial derivatives.

\begin{lemma}
	Under all of the assumptions in Theorem \ref{th}, the global classical solution $(n,{\bf v})$ of the system \eqref{yuan}--\eqref{yuan2} satisfies
\begin{align}
\label{N-11}
\|\nabla^k n\|_{H^{N-k}(\mathbb{R}^d)}^2 + \|\nabla^k{\bf v}\|_{H^{N-k}(\mathbb{R}^d)}^2 \leq C(1+t)^{-\frac{d}{2}-k},
\end{align}
for all $t \geq t_0$ ($t_0$ is a constant to be undetermined) and $k=0,1,2,\ldots,N-1$,
where $C$ is a positive constant independent of time $t$.
\end{lemma}

\begin{proof}  
We will use mathematical induction method and Fourier splitting technique to prove \eqref{N-11}. In fact, the inequality \eqref{0} implies that \eqref{N-11} holds for the case $k=1.$ Therefore, we assume that the decay estimate \eqref{N-11} holds for the case $k=l$,
\begin{align}
\label{k=l}
\|\nabla^l n\|_{H^{N-l}(\mathbb{R}^d)}^2 + \|\nabla^l{\bf v}\|_{H^{N-l}(\mathbb{R}^d)}^2 \leq C(1+t)^{-\frac{d}{2}-l},
\end{align}
where $l=1, 2, 3,\ldots, N-2.$ In what follows, we will prove that \eqref{N-11} also holds for the case $k=l+1.$ In fact, taking $k=l+1$ and $m=N$ in \eqref{leq1}, we can easily obtain
\begin{align}
	\frac{d}{dt}\|\nabla^{l+1} (n,{\bf v})\|_{H^{N-{l-1}}(\mathbb{R}^d)}^2 + C\|\nabla^{l+2} (n,{\bf v})\|_{H^{N-{l-1}}(\mathbb{R}^d)}^2 \leq 0,
\end{align}
which implies
\begin{align}
\label{fn}
\begin{aligned}
	&\frac{d}{dt}\|\nabla^{l+1} (n,{\bf v})\|_{H^{N-{l-1}}(\mathbb{R}^d)}^2 \\
	&+ \frac{C}{2}\big(\|\nabla^{l+2} n\|_{L^2(\mathbb{R}^d)}^2 + \|\nabla^{l+2} n\|_{H^{N-{l-1}}(\mathbb{R}^d)}^2 + \|\nabla^{l+2} {\bf v}\|_{H^{N-{l-1}}(\mathbb{R}^d)}^2\big) \leq 0.
\end{aligned}
\end{align}
Define the time ball
$$
S_0 := {\bigg\{} \zeta \in \mathbb{R}^d {\big |} |\zeta|\leq (\frac{R}{1+t})^{\frac{1}{2}} {\bigg\}},
$$
where $R$ is a constant to be determined. With the help of Parseval's inequality, we have
\begin{equation}
	\begin{aligned}
		\|\nabla^{k+2}n\|_{L^2(\mathbb{R}^d)} & = \int_{\mathbb{R}^d}|\zeta|^{2(k+2)}|\hat{n}|^2d\zeta\\
		&\geq \int_{\mathbb{R}^d/S_0}|\zeta|^{2(k+2)}|\hat{n}|^2d\zeta\\
		&\geq \frac{R}{1+t} \int_{\mathbb{R}^d/S_0}|\zeta|^{2(k+1)}|\hat{n}|^2d\zeta\\
		&\geq \frac{R}{1+t} \int_{\mathbb{R}^d}|\zeta|^{2(k+1)}|\hat{n}|^2d\zeta - \frac{R^2}{(1+t)^2}\int_{S_0}|\zeta|^{2k}|\hat{n}|^2d\zeta\\
		&\geq \frac{R}{1+t} \int_{\mathbb{R}^d}|\zeta|^{2(k+1)}|\hat{n}|^2d\zeta - \frac{R^2}{(1+t)^2}\int_{\mathbb{R}^d}|\zeta|^{2k}|\hat{n}|^2d\zeta.\\
	\end{aligned}
\end{equation}
Then, one has
\begin{equation}
\label{nl}
	\begin{aligned}
		\|\nabla^{l+2} n\|_{L^2(\mathbb{R}^d)}^2 \geq \frac{R}{1+t}\|\nabla^{l+1} n \|_{L^2(\mathbb{R}^d)}^2 - \frac{R^2}{(1+t)^2}\|\nabla^l n\|_{L^2(\mathbb{R}^d)}^2.
	\end{aligned}
\end{equation}
Similarly, one can get
\begin{align}
\label{vk}
	\|\nabla^{k+2}{\bf v} \|_{L^2(\mathbb{R}^d)}^2 \geq \frac{R}{1+t}\|\nabla^{k+1} {\bf v} \|_{L^2(\mathbb{R}^d)}^2 - \frac{R^2}{(1+t)^2}\|\nabla^k {\bf v} \|_{L^2(\mathbb{R}^d)}^2.
\end{align}
By summing the equation \eqref{vk} from $k=l$ to $k=N-1$, we can obtain
\begin{align}
\label{vl}
	\|\nabla^{l+2}{\bf v} \|_{H^{N-l-1}(\mathbb{R}^d)}^2 \geq \frac{R}{1+t}\|\nabla^{l+1} {\bf v} \|_{H^{N-l-1}(\mathbb{R}^d)}^2 - \frac{R^2}{(1+t)^2}\|\nabla^l {\bf v} \|_{H^{N-l-1}(\mathbb{R}^d)}^2.
\end{align}
Substituting equations \eqref{nl} and \eqref{vl} into \eqref{fn} yields that
\begin{equation}
	\begin{aligned}\label{fN}
		&\frac{d}{dt}\|\nabla^{l+1} (n,{\bf v})\|_{H^{N-{l-1}}(\mathbb{R}^d)}^2 + \frac{C}{2}{\big (} \|\nabla^{l+2} n\|_{H^{N-{l-2}}(\mathbb{R}^d)}^2 \\
		&\quad+ \frac{R}{1+t}(\|\nabla^{l+1} n\|_{L^2(\mathbb{R}^d)}^2 + \|\nabla^{l+1} {\bf v}\|_{H^{N-{l-1}}(\mathbb{R}^d)}^2 ){\big )}\\
		&\leq \frac{CR^2}{2(1+t)^2}\big(\|\nabla^ln\|_{L^2(\mathbb{R}^d)}^2 + \|\nabla^l {\bf v}\|_{H^{N-{l-1}}(\mathbb{R}^d)}^2\big).
	\end{aligned}
\end{equation}
For sufficiently large time $t\geq R-1$, one has
\begin{align}
\label{l+2}
	\frac{R}{1+t}\|\nabla^{l+2} n\|_{H^{N-{l-2}}(\mathbb{R}^d)}^2\leq \|\nabla^{l+2} n\|_{H^{N-{l-2}}(\mathbb{R}^d)}^2.
\end{align}
Substituting equations \eqref{l+2} into equation \eqref{fN} yields that
\begin{equation}
	\begin{aligned}
		&\quad\frac{d}{dt}\|\nabla^{l+1} (n,{\bf v})\|_{H^{N-{l-1}}(\mathbb{R}^d)}^2 + \frac{CR}{2(1+t)}(\|\nabla^{l+1} n\|_{H^{N-l-1}(\mathbb{R}^d)}^2 + \|\nabla^{l+1} {\bf v}\|_{H^{N-{l-1}}(\mathbb{R}^d)}^2 {\big )}\\
		&\leq \frac{CR^2}{2(1+t)^2}(\|\nabla^ln\|_{L^2(\mathbb{R}^d)}^2 + \|\nabla^l {\bf v}\|_{H^{N-{l-1}}(\mathbb{R}^d)}^2),
	\end{aligned}
\end{equation}
Together with the definition of $\mathcal{F}_l^m$ in Lemma \ref{lemma 3-3-0} and the time decay estimate \eqref{k=l}, it follows
\begin{align}
\label{7}
	\frac{d}{dt}\mathcal{F}_{l+1}^N(t) + \frac{CR}{2(1+t)}\mathcal{F}_{l+1}^N(t) \leq C(1+t)^{-\frac{d+4+2l}{2}}.
\end{align}
Thus, choosing $$R=\frac{2(l+d)}{C},$$
and multiplying both sides of \eqref{7} by $(1+t)^{l+\frac{d+3}{2}}$, one obtains
\begin{align}
\label{1/2}
	\frac{d}{dt}[(1+t)^{l+\frac{d+3}{2}}\mathcal{F}_{l+1}^N(t)] \leq C(1+t)^{-\frac{1}{2}}.
\end{align}
For any $$t\geq t_0 :=\frac{2(l+\frac{d+3}{2})C_2}{C}-1,$$  integrating \eqref{1/2} on the interval $[0, t]$, we have
\begin{align}
	\mathcal{F}_{l+1}^N(t) \leq [\mathcal{F}_{l+1}^N(0) + C(1+t)^\frac{1}{2}](1+t)^{-(l+\frac{d+3}{2})}.
\end{align}
With the help of the definition of $\mathcal{F}_{l+1}^N$, one obtains
\begin{align}
\label{N-1}
	\|\nabla^{l+1} n\|_{H^{N-{l-1}}(\mathbb{R}^d)}^2 + \|\nabla^{l+1}{\bf v}\|_{H^{N-{l-1}}(\mathbb{R}^d)}^2 \leq C(1+t)^{-\frac{d}{2}-(l+1)},
\end{align}
which implies that \eqref{N-11} holds for the case $k=l+1$.
This lemma has been proved.
\end{proof}   


\subsection{Optimal convergence rates
for critical derivative}

In this subsection, the optimal decay rates for critical spatial derivative ($k=N$) will be given.
To this end, we first show time integrability of the $N$-th order spatial derivative of the solution.
\begin{lemma}
\label{lemma 3.1}
Under all of the assumptions in Theorem \ref{th} , for any fixed constant $0 < \epsilon_0 < 1,$ it holds that
\begin{equation}
		(1+t)^{N-1+\frac{d}{2}}\|\nabla^{N-1} (n,{\bf v})\|_{L^2(\mathbb{R}^d)}^2 + (1+t)^{-\epsilon_0}\int_{0}^{t}(1+\tau)^{N+\frac{d}{2}+\epsilon_0-1}\|\nabla^{N} (n,{\bf v})\|_{L^2(\mathbb{R}^d)}^2d\tau \leq C,
\end{equation}
where $C$ is a positive constant independent of time $t$.
\end{lemma}

\begin{proof} 
By Lemma \ref{lemma 3-3-1}, for $0 \leq k \leq N-1$, we have
\begin{equation}
\label{ef}
	\begin{aligned}
		\frac{d}{dt} \mathcal{E}_k(t) + C_1\mathcal{F}_k(t) \leq 0,
	\end{aligned}
\end{equation}
where
$$\mathcal{E}_k(t) \simeq \|\nabla ^k  (n,{\bf v})\|_{H^1(\mathbb{R}^d)}^2,
 \quad 
 \mathcal{F}_k(t) \simeq \|\nabla ^{k+1} n\|_{H^1(\mathbb{R}^d)}^2 + \|\nabla ^{k+1} {\bf v}\|_{L^2(\mathbb{R}^d)}^2 + \varepsilon\|\nabla ^{k+2} {\bf v}\|_{L^2(\mathbb{R}^d)}^2.
$$
Setting $k=N-1$ in \eqref{ef}, we deduce
\begin{equation*}
\begin{aligned}
\frac{d}{dt} \mathcal{E}_{N-1}(t) + C_1\mathcal{F}_{N-1}(t) \leq 0.
\end{aligned}
\end{equation*}
For any fixed $\epsilon_0$ with $0 < \epsilon_0 < 1$, multiplying the above inequality by $(1+t)^{N+\frac{d}{2}+\epsilon_0-1}$,
one shows that
\begin{equation}
\label{ef2}
\begin{aligned}
\frac{d}{dt} \big[(1+t)^{N+\frac{d}{2}+\epsilon_0-1}\mathcal{E}_{N-1}(t)\big] + C_1(1+t)^{N+\frac{d}{2}+\epsilon_0-1}\mathcal{F}_{N-1}(t) \leq (1+t)^{N+\frac{d}{2}+\epsilon_0-2}\mathcal{E}_{N-1}(t).
\end{aligned}
\end{equation}
Noting the equivalent form of $\mathcal{E}_{N-1}(t)$, we can find that
\begin{equation*}
	\begin{aligned}
		(1+t)^{N+\frac{d}{2}+\epsilon_0-2}\mathcal{E}_{N-1}(t)&\leq C(1+t)^{N+\frac{d}{2}+\epsilon_0-2}\|\nabla ^{N-1} (n,{\bf v})\|_{H^1(\mathbb{R}^d)}^2\\
		& \leq C(1+t)^{N+\frac{d}{2}+\epsilon_0-2}(1+t)^{-(N-1+\frac{d}{2})}\\
		&\leq C(1+t)^{-1+\epsilon_0}.
	\end{aligned}
\end{equation*}
Thus, we have
\begin{equation}
\label{ef3}
	\begin{aligned}
		\frac{d}{dt} \big[(1+t)^{N+\frac{d}{2}+\epsilon_0-1}\mathcal{E}_{N-1}(t)\big] + C_1(1+t)^{N+\frac{d}{2}+\epsilon_0-1}\mathcal{F}_{N-1}(t) \leq C(1+t)^{-1+\epsilon_0}.
	\end{aligned}
\end{equation}
Integrating inequality \eqref{ef3} over $[0,t]$, one achieves
\begin{equation}
\label{ef4}
\begin{aligned}
(1+t)^{N+\frac{d}{2}+\epsilon_0-1}\mathcal{E}_{N-1}(t) + C_1\int_{0}^{t}(1+\tau)^{N+\frac{d}{2}+\epsilon_0-1}\mathcal{F}_{N-1}(\tau)d\tau \leq C(1+t)^{\epsilon_0}.
\end{aligned}
\end{equation}
Together with
$$\mathcal{E}_k(t) \simeq \|\nabla ^k (n,{\bf v})\|_{H^1(\mathbb{R}^d)}^2,  \quad  
\mathcal{F}_k(t) \simeq \|\nabla ^{k+1} n\|_{H^1(\mathbb{R}^d)}^2 + \|\nabla ^{k+1} {\bf v}\|_{L^2(\mathbb{R}^d)}^2 + \varepsilon\|\nabla ^{k+2} {\bf v}\|_{L^2(\mathbb{R}^d)}^2,$$ we arrive at
\begin{equation*}
	\begin{aligned}
&(1+t)^{N+\frac{d}{2}+\epsilon_0-1}\|\nabla^{N-1} (n,{\bf v})\|_{H^1(\mathbb{R}^d)}^2
\\
&\quad + \int_{0}^{t}(1+\tau)^{N+\frac{d}{2}+\epsilon_0-1}(\|\nabla^{N} n\|_{H^1(\mathbb{R}^d)}^2
+ \|\nabla ^{N}{\bf v}\|_{L^2(\mathbb{R}^d)}^2+\varepsilon\|\nabla ^{k+2} {\bf v}\|_{L^2(\mathbb{R}^d)}^2) d\tau
\\
&\leq C(1+t)^{\epsilon_0}.
	\end{aligned}
\end{equation*}	
This completes the proof of this lemma.
\end{proof}  

\begin{lemma}
\label{zuiyou}
Under all of the assumptions in Theorem \ref{th}, the global solution $(n,{\bf v})$ has the following decay estimate:
\begin{equation}
\label{Nyou}
\begin{aligned}
(1&+t)^{N+\frac{d}{2}}\|\nabla^N (n,{\bf v})\|_{L^2(\mathbb{R}^d)}^2
\\
&+(1+t)^{-\epsilon_0}\int_{0}^{t}(1+\tau)^{N+\frac{d}{2}+\epsilon_0}(\|\nabla^{N+1} n\|_{L^2(\mathbb{R}^d)}^2
+\varepsilon\|\nabla^{N+1} {\bf v}\|_{L^2(\mathbb{R}^d)}^2)d\tau \leq C,
\end{aligned}
\end{equation}
where $C$ is a positive constant independent of time $t$.
\end{lemma}

\begin{proof}
Applying the differential operator $\nabla^N$ to \eqref{yuan} and multiplying the resulting equations by $\nabla ^N n, \nabla ^N v$, we obtain
\begin{equation}
\label{37}
	\begin{aligned}
		&\frac{d}{dt}(\|\nabla^N (n,{\bf v})\|_{L^2(\mathbb{R}^d)}^2 + \|\nabla^{N+1} n\|_{L^2(\mathbb{R}^d)}^2 + \varepsilon\|\nabla^{N+1} {\bf v}\|_{L^2(\mathbb{R}^d)}^2) \\
		\leq &C\int_{\mathbb{R}^d} \nabla ^N {\rm div}(n{\bf v})\cdot \nabla^N n dx + C\varepsilon \int_{\mathbb{R}^d} -\nabla ^N (\nabla |{\bf v}|^2)\cdot \nabla^N {\bf v} dx.
	\end{aligned}
\end{equation}
On one hand, it follows from the Gagliardo-Nirenberg inequality and Lemma \ref{lemma 2.2} that
\begin{equation}
\label{38}
	\begin{aligned}
		&\int_{\mathbb{R}^d} \nabla ^N {\rm div}(n{\bf v})\cdot \nabla^N n dx \\
		\leq &C \|\nabla^N (n{\bf v})\|_{L^2(\mathbb{R}^d)}\|\nabla^{N+1} n\|_{L^2(\mathbb{R}^d)}\\
		\leq &C (\|[\nabla^N, n]{\bf v}\|_{L^2(\mathbb{R}^d)}+\|n\cdot\nabla^N {\bf v}\|_{L^2(\mathbb{R}^d)})\|\nabla^{N+1} n\|_{L^2(\mathbb{R}^d)}\\
		\leq &C(\|\nabla n\|_{L^\infty(\mathbb{R}^d)}\|\nabla^{N-1}{\bf v}\|_{L^2(\mathbb{R}^d)} + \|\nabla^N n\|_{L^2(\mathbb{R}^d)}\|{\bf v}\|_{L^\infty(\mathbb{R}^d)} \\
&+ \|\nabla^N {\bf v}\|_{L^2(\mathbb{R}^d)}\|n\|_{L^\infty(\mathbb{R}^d)})\|\nabla^{N+1} n\|_{L^2(\mathbb{R}^d)}\\
		\leq &\delta \|\nabla^{N+1} n\|_{L^2(\mathbb{R}^d)}^2 + C_{\delta}\|\nabla n\|_{L^\infty(\mathbb{R}^d)}^2\|\nabla^{N-1}{\bf v}\|_{L^2(\mathbb{R}^d)}^2 \\
		&+ C_{\delta}\|(n,{\bf v})\|_{L^\infty(\mathbb{R}^d)}^2\|\nabla ^N (n,{\bf v})\|_{L^2(\mathbb{R}^d)}^2.\\
	\end{aligned}
\end{equation}
For $d=2$, one has
\begin{equation}
	\begin{aligned}
		\|\nabla n\|_{L^\infty(\mathbb{R}^2)}\|\nabla^{N-1}{\bf v}\|_{L^2(\mathbb{R}^2)}\leq \|\nabla n\|_{L^2(\mathbb{R}^2)}^{\frac{1}{2}} \|\nabla^3 n\|_{L^2(\mathbb{R}^2)}^{\frac{1}{2}}\|\nabla^{N-1}{\bf v}\|_{L^2(\mathbb{R}^2)}\leq(1+t)^{-\frac{N+1+2}{2}},
	\end{aligned}
\end{equation}
while for $d=3$, we can find that
\begin{equation}
	\begin{aligned}
		\|\nabla n\|_{L^\infty(\mathbb{R}^3)}\|\nabla^{N-1}{\bf v}\|_{L^2(\mathbb{R}^3)}\leq  \|\nabla^2 n\|_{H^1(\mathbb{R}^3)}^{\frac{1}{2}}\|\nabla^{N-1}{\bf v}\|_{L^2(\mathbb{R}^3)}\leq(1+t)^{-\frac{N+1+3}{2}}.
	\end{aligned}
\end{equation}
Based on the above results, we conclude that
$$
\|(n,{\bf v})\|_{L^\infty(\mathbb{R}^d)}\|\nabla ^N (n,{\bf v})\|_{L^2(\mathbb{R}^d)} \leq (1+t)^{-\frac{1+\frac{d}{2}}{2}}\|\nabla ^N (n,{\bf v})\|_{L^2(\mathbb{R}^d)}.
$$
Therefore, we assert that
\begin{equation}
	\begin{aligned}
		&\int_{\mathbb{R}^d} \nabla ^N {\rm div}(n{\bf v})\cdot \nabla^N n dx \\
		\leq &\delta \|\nabla^{N+1} n\|_{L^2(\mathbb{R}^d)}^2 + C_{\delta}(1+t)^{-(1+\frac{d}{2})}\|\nabla ^N (n,{\bf v})\|_{L^2(\mathbb{R}^d)}^2 + C_{\delta} (1+t)^{-(N+1+d)}.
	\end{aligned}
\end{equation}
On the other hand, let us estimate the second term on the right hand of \eqref{37},
\begin{equation}
	\begin{aligned}\label{39}
		&-\varepsilon\int_{\mathbb{R}^d}\nabla ^N (\nabla |{\bf v}|^2)\cdot \nabla^N {\bf v} dx
 \\
 &=\varepsilon  \int_{\mathbb{R}^d} {\rm div}\cdot \nabla ^N {\bf v} \cdot \nabla^N {\bf v}^2 dx\\
		&\leq C\varepsilon  \|\nabla^{N+1}{\bf v}\|_{L^2(\mathbb{R}^d)}\|\nabla ^N {\bf v}^2\|_{L^2(\mathbb{R}^d)}\\
		&\leq C\varepsilon  (\|[\nabla^N, {\bf v}]{\bf v}\|_{L^2(\mathbb{R}^d)}+\|{\bf v}\cdot\nabla^N {\bf v}\|_{L^2(\mathbb{R}^d)})\|\nabla^{N+1}{\bf v}\|_{L^2(\mathbb{R}^d)}\\
		&\leq  \varepsilon\delta \|\nabla^{N+1}{\bf v}\|_{L^2(\mathbb{R}^d)}^2 +  \varepsilon C_{\delta}(1+t)^{-(1+\frac{d}{2})}\|\nabla ^N {\bf v}\|_{L^2(\mathbb{R}^d)}^2 +  \varepsilon C_{\delta}(1+t)^{-(N+1+d)}.
	\end{aligned}
\end{equation}
Substituting \eqref{38} and \eqref{39} into \eqref{37} and choosing $\epsilon$ and $\delta$ sufficiently small, we conclude that
\begin{equation}
\label{3.57}
	\begin{aligned}
		&\frac{d}{dt}\|\nabla^N (n,{\bf v})\|_{L^2(\mathbb{R}^d)}^2 + \|\nabla^{N+1} n\|_{L^2(\mathbb{R}^d)}^2 + \varepsilon\|\nabla^{N+1} {\bf v}\|_{L^2(\mathbb{R}^d)}^2 \\
		\leq &\delta \|\nabla^{N+1} n\|_{L^2(\mathbb{R}^d)}^2 + \varepsilon \delta \|\nabla^{N+1}{\bf v}\|_{L^2(\mathbb{R}^d)}^2 \\
		&+  C_{\delta}(1+t)^{-1}\|\nabla ^N (n,{\bf v})\|_{L^2(\mathbb{R}^d)}^2 + C_{\delta} (1+t)^{-(N+1+d)}.
	\end{aligned}
\end{equation}
Multiplying \eqref{3.57} by $(1+t)^{N+\frac{d}{2}+\epsilon_0}$, integrating the resulting inequality respect to time $t$ and using Lemma \ref{lemma 3.1}, one can get
\begin{equation*}
	\begin{aligned}
		&(1+t)^{N+\frac{d}{2}+\epsilon_0}\|\nabla ^N (n,{\bf v})\|_{L^2(\mathbb{R}^d)}^2 + \int_{0}^{t}(1+\tau)^{N+\frac{d}{2}+\epsilon_0}(\|\nabla ^{N+1} n\|_{L^2(\mathbb{R}^d)}^2 +\varepsilon \|\nabla ^{N+1}{\bf v}\|_{L^2(\mathbb{R}^d)}^2)d\tau\\
		\leq &C\|\nabla^N (n_0,{\bf v}_0)\|_{L^2(\mathbb{R}^d)}^2 + C\int_{0}^{t}(1+\tau)^{N+\frac{d}{2}+\epsilon_0-1}\|\nabla^N (n,{\bf v})\|_{L^2(\mathbb{R}^d)}^2d\tau + C\int_{0}^{t}(1+\tau)^{-1+\epsilon_0}d\tau\\
		\leq &C(1+\|\nabla^N (n_0,{\bf v}_0)\|_{L^2(\mathbb{R}^d)}^2) + C(1+t)^{\epsilon_0}.
	\end{aligned}
\end{equation*}
which deduces the desired estimate directly.
The proof of this lemma is completed.
\end{proof}


\subsection{Lower bounds of the convergence rates}

In this subsection, we want to establish lower bounds of decay rates for the global solution to the initial value problem \eqref{yuan}. For this purpose, we need to discuss the upper bounds of decay rates for $(\underline{n}, \underline{\bf v})$ (defined below) and its first-order spatial derivative.

Define $(\underline{n}, \underline{\bf v}) := (n-\widetilde{n}, {\bf v}-\widetilde{{\bf v}}),$ then
$(\underline{n}, \underline{\bf v})$ satisfies
\begin{equation}
\label{feixian}
	\left\{
	\begin{array}{l}
		\underline{n}_t - \Delta \underline{n} - {\rm div}\underline{\bf v} = S_1, \\
		\underline{\bf v}_{t} -  \Delta \underline{\bf v} - \nabla \underline{n} = S_2,\\
		(\underline{n},\underline{\bf v})(x,0)=(0,0).
	\end{array}
	\right.
\end{equation}

\begin{lemma}
	Under all of the assumptions in Theorem \ref{th}, it holds that for $t \geq 0$,
\begin{equation}
\label{feixianshang}
\begin{aligned}
\|\nabla ^k(\underline{n}, \underline{\bf v})(t) \|_{L^2(\mathbb{R}^d)} \leq \widetilde{C}\gamma (1+t)^{-\frac{d}{4}-\frac{k}{2}},
\quad \text{for}\  k=0,1,
\end{aligned}
\end{equation}
where $\gamma=\min\{\kappa, \kappa^\frac12M_0^{\frac12}\}$ and $\widetilde{C}$ is a constant independent of time $t$.
\end{lemma}

\begin{proof}
By Duhamel's principle, it shows that,  for $k \geq 0$,
\begin{equation}
\label{delta}
	\begin{aligned}
		\|\nabla^k(\underline{n}, \underline{\bf v})\|_{L^2(\mathbb{R}^d)} \leq \int_{0}^{t}(1+t-\tau)^{-\frac{d}{4}-\frac{k}{2}}(\|(S_1, S_2)(\tau)\|_{L^1(\mathbb{R}^d)} + \|\nabla ^k(S_1, S_2)(\tau) \|_{L^2(\mathbb{R}^d)})d\tau.
	\end{aligned}
\end{equation}
Then, from the decay estimates \eqref{S} and \eqref{N-11}, it follows that
\begin{equation}
\label{SS}
	\begin{aligned}
		\|(S_1,S_2)\|_{L^1(\mathbb{R}^d)} \leq C\gamma(1+t)^{-\frac{d+2}{4}},\\
		\|(S_1,S_2)\|_{L^2(\mathbb{R}^d)} \leq C\gamma(1+t)^{-\frac{d+2}{4}},\\
		\|\nabla (S_1,S_2)\|_{L^2(\mathbb{R}^d)} \leq C\gamma(1+t)^{-\frac{d+4}{4}}.\\
	\end{aligned}
\end{equation}
Then, substituting the decay estimates \eqref{SS} into \eqref{delta} with $k=0$ and $k=1 $ respectively, one obtains that
\begin{equation*}
	\begin{aligned}
		\|(\underline{n}, \underline{\bf v})\|_{L^2(\mathbb{R}^d)} &\leq \int_{0}^{t}(1+t-\tau)^{-\frac{d}{4}}(\|(S_1, S_2)(\tau)\|_{L^1} + \|(S_1, S_2)(\tau) \|_{L^2(\mathbb{R}^d)})d\tau\\
		&\leq C\gamma\int_{0}^{t}(1+t-\tau)^{-\frac{d}{4}}(1+\tau)^{-\frac{d+2}{4}}d\tau\\
		&\leq C \gamma(1+t)^{-\frac{d}{4}},
	\end{aligned}
\end{equation*}
	  and
	  \begin{align*}
	  	\|\nabla(\underline{n}, \underline{\bf v})\|_{L^2(\mathbb{R}^d)} &\leq \int_{0}^{t}(1+t-\tau)^{-\frac{d+2}{4}}(\|(S_1, S_2)(\tau)\|_{L^1(\mathbb{R}^d)} + \|\nabla (S_1, S_2)(\tau) \|_{L^2(\mathbb{R}^d)})d\tau\\
	  	&\leq \gamma\int_{0}^{t}(1+t-\tau)^{-\frac{d+2}{4}}(1+\tau)^{-\frac{d+4}{4}}d\tau\\
	  	&\leq C\gamma(1+t)^{-\frac{d+2}{4}}.
	  \end{align*}
Therefore, the proof of this lemma is completed.
\end{proof} 

Next, we establish the lower bounds of decay rates for the global solution and its spatial derivatives of system \eqref{yuan}.

\begin{lemma}
	Under all of the assumptions in Theorem \ref{th1},   we have, for any $t$ large enough,
\begin{equation}
\label{zong}
		\min\{\|\nabla ^k n(t)\|_{L^2(\mathbb{R}^d)}, \|\nabla ^k {\bf v}(t)\|_{L^2(\mathbb{R}^d)}\} \geq \tilde{c} (1+t)^{-\frac{d}{4}-\frac{k}{2}}, \quad \text{for}\quad  0\leq k \leq N,
\end{equation}
where $\tilde{c}$  is a positive constant independent of time $t$.
\end{lemma}
\begin{proof}
By virtue of the definition of $\underline{n}$, we have
$$
\|\nabla ^k \tilde{n}\|_{L^2(\mathbb{R}^d)} \leq \|\nabla ^k n \|_{L^2(\mathbb{R}^d)} + \|\nabla ^k \underline{n}\|_{L^2(\mathbb{R}^d)},
$$
which, together with \eqref{nv} and \eqref{feixianshang},  yields directly
\begin{equation}
\label{01}
	\|\nabla ^k n \|_{L^2(\mathbb{R}^d)} \geq \|\nabla ^k \tilde{n}\|_{L^2(\mathbb{R}^d)} - \|\nabla ^k \underline{n}\|_{L^2(\mathbb{R}^d)} \geq \tilde{c}(1+t)^{-\frac{d}{4}-\frac{k}{2}}- \tilde{C}\gamma (1+t)^{-\frac{d}{4}-\frac{k}{2}},
\end{equation}
where $k=0,1.$ It is worth noting that the small constant $\gamma$ is used to control the upper bound of initial data in $L^2$ -norm instead of $L^1$ one . From the estimate \eqref{nv} in Lemma \ref{linear}, the constant $\tilde{c}$ in
\eqref{01} only depends on $M_n$ and $M_v$. Then, we can choose $\gamma$ sufficiently small such that $\tilde{C}\gamma \leq \frac{1}{2}\tilde{c}$, and it follows from \eqref{01} that
\begin{equation}
\label{k01}
	\|\nabla ^k n \|_{L^2(\mathbb{R}^d)} \geq \frac{1}{2}\tilde{c}(1+t)^{-\frac{d}{4}-\frac{k}{2}}, \quad k=0,1.
\end{equation}
By virtue of the Sobolev interpolation inequality in Lemma \ref{lemma 2.1}, it holds that, for $k \geq 2$,
$$
\|\nabla n\|_{L^2(\mathbb{R}^d)} \leq C\| n\|_{L^2(\mathbb{R}^d)}^{1-\frac{1}{k}}\|\nabla^k n\|_{L^2(\mathbb{R}^d)}^{\frac{1}{k}},
$$
which, together with the lower bounds of decay rates \eqref{k01} and upper bounds of decay rates \eqref{N-1}\eqref{Nyou},  implies that
\begin{equation}
\label{kn}
	\|\nabla ^k n \|_{L^2(\mathbb{R}^d)} \geq C \|\nabla n \|_{L^2(\mathbb{R}^d)}^k\|n\|_{L^2(\mathbb{R}^d)}^{-(k-1)} \geq C(1+t)^{-\frac{(d+2)k}{4}}(1+t)^{\frac{d(k-1)}{4}} \geq \tilde{c} (1+t)^{-\frac{d}{4}-\frac{k}{2}},
\end{equation}
for $k\geq 2$. Similarly, it is easy to deduce that
\begin{equation}
\label{kv}
	\|\nabla ^k {\bf v} \|_{L^2(\mathbb{R}^d)}  \geq \tilde{c} (1+t)^{-\frac{d}{4}-\frac{k}{2}}, \quad for \quad k\geq 0.
\end{equation}
Then, the combination of estimates \eqref{k01}, \eqref{kn} and \eqref{kv} yields the estimate \eqref{zong}. Therefore, we complete the proof of this lemma.
\end{proof}

\subsection{The proof of Theorem \ref{th} and Theorem \ref{th1}}
Under all of the assumptions in Theorem \ref{th} and Theorem \ref{th1}, with the help of estimations \eqref{nv}, \eqref{N-11}, \eqref{Nyou} and \eqref{zong}, we arrive at
\begin{align}
\label{3.67}
	\tilde{c}(1+t)^{-\frac{d+2k}{4}}\leq	\Vert \nabla ^k(u-\bar{u},{\bf v})(t) \Vert _{L^2(\mathbb{R}^d)} \leq C(1+t)^{-\frac{d+2k}{4}}, \quad for \quad k=0, 1, \ldots, N,
\end{align}
where $\tilde{c}$ and $C$ are positive constants independent of time. The Proof of Theorem \ref{th} and Theorem \ref{th1} is completed. 
Thus, we obtain the decay rates of the solution, which are the same as the decay rates of the linear system. Therefore, they are shown to be optimal.


\section{Convergence Rate of the Original System}

In order to prove Theorem \ref{th2}, we can discuss in the following steps. Firstly, noticing that the function $u$ keeps the same in both the original model \eqref{1.1} and the transformed system \eqref{1.2}, it is straightforward to establish the results for $u$. Secondly, using the Cole-Hopf transformation, 
we can prove the existence of global classical solutions of \eqref{1.1} by applying Proposition \ref{prop} and parabolic regularity theory. Finally, based on \eqref{thn} and the Gagliardo-Nirenberg inequality for dimensions $d=2,3$, we can derive the time convergence rate for $u$ as
\begin{equation}
\label{4.1}
	\begin{aligned}
		\|u-\bar{u}\|_{L^\infty(\mathbb{R}^d)}&=\|n\|_{L^\infty(\mathbb{R}^d)}\\
		                       &\leq C\|n\|_{L^2(\mathbb{R}^d)}^\frac{1}{2}\|\nabla^d n\|_{L^2(\mathbb{R}^d)}^\frac{1}{2}\\
		                       &\leq C\|n\|_{L^2(\mathbb{R}^d)}^\frac{1}{2}\|\nabla^2 n\|_{L^(\mathbb{R}^d)}^\frac{1}{2}\\
		                       &\leq C  (1+t)^{-\frac{d+2}{4}},
	\end{aligned}
\end{equation}
where the decay rates of the higher derivatives for the solution in Theorem \ref{th} are used.

To investigate the decay rate of the chemical concentration $c$, the second equation of \eqref{1.1} and the Cole-Hopf transformation 
yield the following equality
$$
({\rm ln}c)_t = -\varepsilon\nabla{\bf v} + \varepsilon{\bf v}^2 - u,
$$
which, together with integration, gives
\begin{equation}
	\begin{aligned}\label{c}
		c(x,t)=c_0(x){\rm exp}(-\bar{u}t + \int_0^t(\bar{u}-u + \varepsilon({\bf v}^2-\nabla {\bf v}))d\tau).
	\end{aligned}
\end{equation}
It follows from \eqref{thn} and Lemma \ref{lemma 2.1} that
\begin{equation}
	\begin{aligned}\label{u}
		\int_0^t \|u-\bar{u}\|_{L^\infty(\mathbb{R}^d)}d\tau &\leq C	\int_0^t \|n\|_{L^2(\mathbb{R}^d)}^\frac{1}{2}\|\nabla^d n\|_{L^2(\mathbb{R}^d)}^\frac{1}{2}d\tau\\
		&\leq C	\int_0^t \|n\|_{L^2(\mathbb{R}^d)}^\frac{1}{2}\|\nabla^2 n\|_{L^(\mathbb{R}^d)}^\frac{1}{2}d\tau\\
		&\leq C	\int_0^t (1+\tau)^{-\frac{d+2}{4}}d\tau\\
			&\leq C(1+t)^{\frac{2-d}{4}}.
	\end{aligned}
\end{equation}
Similarly, it concludes
\begin{equation}
	\begin{aligned}
		\int_0^t \|{\bf v}^2\|_{L^\infty(\mathbb{R}^d)}d\tau \leq C(1+t)^{-\frac{d}{2}}
	\end{aligned}
\end{equation}
and
\begin{equation}
	\begin{aligned}\label{v2}
		\int_0^t \|\nabla{\bf v}\|_{L^\infty(\mathbb{R}^d)}d\tau \leq C(1+t)^{\frac{2-d}{4}}.
	\end{aligned}
\end{equation}
Substituting \eqref{u}-\eqref{v2} into \eqref{c} yields
\begin{equation}
\label{4.6}
	\|c\|_{L^\infty(\mathbb{R}^d)} \leq Ce^{-\bar{u}(1+t)[1-C(1+t)^{-\frac{2+d}{4}}-C(1+t)^{-1-\frac{d}{2}}]}\leq Ce^{-\bar{u}t}.
\end{equation}
The conclusion follows from \eqref{4.1} and \eqref{4.6}.
\hfill$\Box$

\bigskip

\section*{Acknowledgement}
{\small Q. Tao was supported in part by the National Natural Science Foundation of China (12471423) and the Guangdong Basic and Applied
Basic Research Foundation (2025A1515010485). D. Wang was supported in part by NSF grants DMS-2219384 and DMS-2510532.
Y. Yang  was supported in part by the National Natural Science Foundation of China (12271186), the Shenzhen Science and Technology Program (JCYJ20241202124209011) and the Shenzhen
Natural Science Fund (the Stable Support Plan Program 20231122110251002).
}

\section*{Data Availability}
No data was used for the research described in the article.

\section*{Declarations}
Conflict of interest The authors declare that they have no Conflict of interest.

\end{document}